\documentclass{amsart}

\usepackage{amssymb}
\usepackage{graphicx}
\usepackage{MnSymbol}

\usepackage{amsthm}

\title{On definable subgroups of the fundamental group}
\author{Samuel M. Corson}

\theoremstyle{definition}\newtheorem*{A}{Theorem \ref{BigTheorem}}
\theoremstyle{definition}\newtheorem*{B}{Theorem \ref{ascendingnormal}}
\theoremstyle{definition}\newtheorem*{C}{Theorem \ref{freeprodPolish}}

\theoremstyle{definition}\newtheorem{bigtheorem}{Theorem}

\theoremstyle{definition}\newtheorem{theorem}{Theorem}[section]

\theoremstyle{definition}\newtheorem{corollary}[theorem]{Corollary}
\theoremstyle{definition}\newtheorem{proposition}[theorem]{Proposition}
\theoremstyle{definition}\newtheorem{definition}[theorem]{Definition}
\theoremstyle{definition}\newtheorem{question}[theorem]{Question}
\theoremstyle{definition}\newtheorem{example}{Example}
\theoremstyle{definition}\newtheorem{remark}[theorem]{Remark}
\theoremstyle{definition}
\theoremstyle{definition}\newtheorem{lemma}[theorem]{Lemma}
\theoremstyle{definition}
\theoremstyle{definition}
\theoremstyle{definition}
\theoremstyle{definition}
\newtheorem*{question*}{Question}
\newtheorem*{theorem*}{Theorem}
\newtheorem*{corollary*}{Corollary}
\newtheorem*{lemma*}{Lemma}

\def\pmc#1{\setbox0=\hbox{#1}
    \kern-.1em\copy0\kern-\wd0
    \kern.1em\copy0\kern-\wd0}

\newcommand{\diam}{\operatorname{diam}}
\newcommand{\Star}{\operatorname{Star}}

\newcommand{\W}{\mathcal{W}}
\newcommand{\HEG}{\textbf{HEG}}
\newcommand{\Po}{\mathcal{P}}
\newcommand{\card}{\operatorname{card}}

\begin{document}


\email{sammyc973@gmail.com}

\maketitle

\begin{abstract}  We present several new theorems concerning the first fundamental group of a path connected metric space.   Among the results proven are strengthenings of the main theorems of \cite{Sh2} and \cite{CoCo}.   A compactness theorem for the fundamental group of a Peano continuum is given.  A useful characterization for the shape kernel of a locally path connected space is presented, along with a very succinct proof of the fact that for such a space the Spanier and shape kernel subgroups coincide (see \cite{BF}).  We also show that a free decomposition of the fundamental group of a locally path connected Polish space cannot be nonconstructive.  Numerous other results and examples illustrating the sharpness of our theorems are provided.

\end{abstract}

\begin{section}{Introduction}  One way to understand a path connected topological  space is to analyze its fundamental group.  Fundamental groups are a homotopy invariant and provide a useful tool for distinguishing homotopy equivalence classes.  In understanding the fundamental group it is useful to study subgroups that are definable in terms of topology or logic or some combination of the two.  The focus of this paper is the study of the first fundamental group of metric spaces, and certain of its subgroups.  Assumptions about separability and generalizations of separability, local path connectedness, and compactness will figure prominently in our study.  Throughout this paper we take the simplifying assumption that all spaces for which a fundamental group is computed are path connected.

In Section \ref{Prelim} we give preliminary definitions and then define the characterization of subgroups of the fundamental group via the topology of the loop space.  For $G \leq \pi_1(X, x)$ we study the relationship of  $G$ to $\pi_1(X, x)$ by looking at how $\bigcup G$ sits in the loop space $L_x$.  This very simple idea yields a diversity of theorems.  The theory of open and closed subgroups which is developed in this section gives proofs of many results later in the paper (e.g. Theorems \ref{shapekernel},  \ref{Spaniertheorem}, \ref{nosurjection}; Corollary \ref{imageinnslender}).

In Section \ref{Polishspaces} we introduce concepts related to separable completely metrizable spaces- Polish spaces.  Some terminology and tools of descriptive set theory are introduced, including the class of analytic (denoted $\Sigma_1^1$) sets and generalizations thereof (what we call nice classes of sets).  The class of subgroups of type $\Sigma_1^1$ enjoys many closure properties (detailed in Theorem \ref{closureprop}).  Certain techniques of descriptive set theory yield new theorems regarding the fundamental group and some of its quotients (see Theorems \ref{unc} and \ref{Shelahgen}), of which the following gives a short catalogue of applications:

\begin{bigtheorem}\label{BigTheorem}  Suppose $X$ is a locally path connected, connected Polish space.  The following groups are of cardinality $2^{\aleph_0}$ or $\leq \aleph_0$, and in case $X$ is compact they are of cardinality $2^{\aleph_0}$ or are finitely generated:

\begin{enumerate}  \item $\pi_1(X)$

\item $\pi_1(X)/(\pi_1(X))^{(\alpha)}$ for any $\alpha<\omega_1$ (derived series)

\item $\pi_1(X)/(\pi_1(X))_n$ for any $n\in \omega$ (lower central series)

\item $\pi_1(X)/N$ where $N$ is the normal subgroup generated by squares of elements, cubes of elements, or $n$-th powers of elements.
\end{enumerate}

\noindent In case $X$ is compact then countability of the fundamental group is equivalent to being finitely presented.
\end{bigtheorem}

The compact case of part (1) is the main result of \cite{Sh2} and part (3) with $n=1$ is the main result of \cite{CoCo}.  In Section \ref{Polishspaces} we also give a compactness-type theorem for Peano continua (the reader may subsitute $\Sigma_1^1$ for $\Po$):

\begin{bigtheorem}\label{ascendingnormal}  If $X$ is a Peano continuum there does not exist a strictly increasing infinite sequence of $\Po$ normal subgroups $\{G_n\}_{n\in \omega}$ of $\pi_1(X)$ such that $\bigcup_{n\in \omega} G_n = \pi_1(X)$.
\end{bigtheorem}

In Section \ref{comonstergroups} we present an application of Theorem \ref{ascendingnormal}.  We define a comonster group to be a group which is not the normal subgroup closure of any finite subset and a $\kappa$-comonster group is not the normal subgroup closure of any set of cardinality $<\kappa$.  Theorem \ref{ascendingnormal} implies that if the fundamental group of a Peano continuum is co-monster then it is $\aleph_1$-comonster (Theorem \ref{comonster}).  Examples are presented of this phenomenon and a curious tie to finitely presented groups is drawn (Theorem \ref{finitelypresentedgroups}).

In Section \ref{Someexamples} we compute the complexity of some commonly used subgroups of the fundamental group.  The shape kernel is shown to be closed, and if the space is locally path connected the shape kernel is the intersection of all clopen subgroups (Theorem \ref{shapekernel}).  The Spanier subgroup is shown to be equal to the shape kernel in case the space is locally path connected, and is shown to be $\Sigma_1^1$ in case the space is compact (Theorem \ref{Spaniertheorem}).  That the Spanier and shape kernel subgroups coincide for locally path connected paracompact Hausdorff spaces is a recent theorem of Brazas and Fabel \cite{BF}.  Though our theorem is slightly less general, the proof is rather shorter than that of \cite{BF}.

In Section \ref{freeprod} we introduce n-slender groups (see \cite{E1}), give such groups an alternative characterization, and present some theorems using the theory of open and closed subgroups.  Among the results of the section is the fact that a locally path connected separable metric space cannot have fundamental group that is an uncountable free product of nontrivial groups (Corollary \ref{nouncountabledecomp}).  We also prove the following: 

\begin{bigtheorem}  \label{freeprodPolish}  Suppose $X$ is locally path connected Polish and $\pi_1(X) \simeq *_{i\in I}G_i$ with each $G_i$ nontrivial.  The following hold:

\begin{enumerate}\item$\card(I)\leq \aleph_0$

\item Each retraction map $r_j:*_{i\in I}G_i \rightarrow G_j$ has analytic kernel.

\item  Each $G_j$ is of cardinality $\leq \aleph_0$ or $2^{\aleph_0}$.

\item  The map $ *_{i\in I}G_i \rightarrow \bigoplus_{i\in I}G_i$ has analytic kernel.

\end{enumerate}
\end{bigtheorem}

This theorem can be interpreted to mean that no free decomposition of a fundamental group as in the hypotheses can be non-constructive.  This is rather surprising in light of the fact that a direct sum decomposition of the fundamental group can be non-constructive (see discussion in Section \ref{freeprod}).

In Section \ref{conclusion} we give a brief discussion of what are called nice pointclasses (these agree with the $\Po$ that is found in some of the theorems stated so far).  We discuss which pointclasses can be assumed to be nice (consistent with the standard axioms of set theory), and hence to which subgroups we can consistently apply the theorems of Sections \ref{Polishspaces} and \ref{comonstergroups}.

Our discussions  will not avoid references to abstract combinatorial set theory.  This should not be surprising as topology is `visual set theory.'  Also, many facets of descriptive set theory are directly influenced by the model of set theory in which one is working.  We will keep our references to set theory simple, doing little beyond illustrating the sharpness of our results until the discussion in Section \ref{conclusion}.

Let \textbf{ZFC} denote the Zermelo-Fraenkel axioms of set theory including the axiom of choice.  We assume \textbf{ZFC} throughout this paper, and also assume that \textbf{ZFC} is consistent so as to avoid repetition of the phase ``if \textbf{ZFC} is consistent then there exists a model $\ldots$''  The parameter $\kappa$ will be used for infinite cardinals, $\kappa^+$ denotes the successor cardinal.  Let \textbf{CH} denote the continuum hypothesis: $\aleph_0^+ = 2^{\aleph_0}$, and \textbf{GCH} denote the generalized continuum hypothesis: $(\forall \kappa)[\kappa^+ = 2^{\kappa}]$.
\end{section}

\begin{section}{Preliminaries}\label{Prelim}

In this section we present some of the basic definitions and notation for fundamental groups in this paper.  We then give some lemmas about open and closed subgroups of the fundamental group which will be used throughout.

Given a topological space $X$ and distinguished point $x\in X$ we obtain the fundamental group $\pi_1(X, x)$ as follows.  A loop based at $x$ is a continuous function $l: ([0,1], \{0,1\}) \rightarrow (X, x)$.  Two loops $l_0$ and $l_1$ at $x$ are homotopic if there exists a continuous function $H:[0,1]\times [0,1] \rightarrow X$ called a homotopy such that $H(s, 0) = l_0(s)$, $H(s, 1) = l_1(s)$ and $H(0, t) = H(1, t) = x$ for all $s, t\in [0,1]$.  The relation defined by homotopy is an equivalence relation.  Letting $L_x$ denote the set of all loops at $x$ in $X$ we have the binary operation concatenation, denoted $*$, on $L_x$ defined by $l_0*l_1(s) = \begin{cases}l_0(2s)$ if $s\in [0, \frac{1}{2}]\\l_1(2s-1)$ if $s\in [\frac{1}{2}, 1]  \end{cases}$.  This definition also works as a partial binary operation on paths, defined whenever the first path ends where the second path begins.  For specificity, we mean $l_0 * (l_1 *(\cdots *(l_{n-1}*l_n)\cdots)$ when we write $l_0*l_1*\cdots *l_n$.  There is also a unary operation $^{-1}$ given by $l^{-1}(s) = l(1-s)$.  The fundamental group is the set $L_x$ modulo homotopy, the binary operation is given by $[l_0]*[l_1] = [l_0*l_1]$, the equivalence class of the constant loop is the identity and inverses are given by $[l]^{-1} = [l^{-1}]$.  Clearly the fundamental group $\pi_1(X, x)$ is the same as the fundamental group of $\pi_1(C, x)$ where $C$ is the path component of $x$.  \textit{We shall only consider fundamental groups of spaces which are path connected.}

We assume some familiarity with notions in topology such as metrizability and separability.  Let $Z$ be a topological space.  A \textbf{pointclass} is a collection $\Po$ of subsets of $Z$ that are of a particular topological description, usually in terms of countable unions, countable intersections, complements, or projections.  For example, the collection of open subsets (topology) of $Z$, the collection of closed sets of $Z$, and the collection of countable unions of closed sets of $Z$ are all pointclasses of $Z$.  Another example is the class of Borel subsets of $Z$.  When we restrict our attention to specific types of topological spaces, we get more information about sets in pointclasses.

Take $(X, d)$ to be a metric space with distinguished point $x\in X$.  Topologize $L_x$ by the $\sup$ metric: the distance between loops $l_0$ and $l_1$ is $\sup_{s\in [0,1]} d(l_0(s),l_1(s))$.  Since uniform convergence is equivalent to convergence in the compact-open topology, we may suppress the particular metric $d$ on the space $X$ (since any other compatible metric gives the same topology on $L_x$).

\begin{definition}  A subgroup $G \leq \pi_1(X, x)$ is of pointclass $\Po$ if the collection of loops belonging to elements of $G$ is in the pointclass $\Po$ in $L_x$.  In other words, $G\leq \pi_1(X, x)$ is of pointclass $\Po$ if $\bigcup G$ is in pointclass $\Po$ in $L_x$.
\end{definition}

We establish some lemmas.  Lemmas \ref{supergroup}, \ref{relations}, and \ref{openclosed} should remind the reader of the analogous facts for topological groups.

\begin{lemma}\label{supergroup}  If $G\leq \pi_1(X, x)$ is open and $G \leq H \leq \pi_1(X, x)$ then $H$ is open.
\end{lemma}

\begin{proof}  Let $G$ be open and $l\in \bigcup H$ with $\{l_n\}_{n\in \omega}$ a sequence in $L_x$ converging to $l$.  Since $l*l^{-1}\in \bigcup G$ there exists $\epsilon>0$ such that $B(l*l^{-1}, \epsilon) \subseteq \bigcup G$.  The sequence $\{l*l_n^{-1}\}_{n\in \omega}$ is eventually in $B(l*l^{-1}, \epsilon)$, so that $\{l*l_n^{-1}\}_{n\in \omega}$ is eventually in $\bigcup G \subset \bigcup H$, so $\{l_n^{-1}\}_{n\in \omega}$ is eventually in $\bigcup H$, so $\{l_n\}_{n\in \omega}$ is eventually in $\bigcup H$.
\end{proof}

\begin{lemma}\label{relations}  If $\Po$ is closed under continuous preimages and $H \leq \pi_1(X, x)$ is $\Po$ then:

\begin{enumerate}

\item The equivalence relations $E,R \subseteq L_x\times L_x$ defined by $l_0 E l_1$ iff $[l_0]H = [l_1]H$ and $l_0 R l_1$ iff $H[l_0] = H[l_1]$ are $\Po$.

\item Each equivalence class in $E$ and $R$ is $\Po$.

\end{enumerate}

\noindent By $[l]H$ we mean the set of all loops based at $x$ which are homotopic to a loop of the form $l \ast l'$ where $l'\in \bigcup H$ and the definition for $H[l]$ is analogous.

\end{lemma}

\begin{proof}  The function $L_x \times L_x \rightarrow L_x$ given by $(l_0, l_1) \mapsto (l_0)^{-1}*l_1$ is continuous and $E$ is the preimage of $\bigcup H$ under this function, so by assumption we have $E$ is $\Po$.  The proof that $R$ is $\Po$ is similar.  This proves (1).  For (2) we notice that for a fixed $l_0 \in L_x$ the function $L_x \rightarrow L_x$ given by $l \mapsto (l_0)^{-1} * l$ is continuous and the set $[l_0]H$ is the continuous preimage of $\bigcup H$.
\end{proof}

\begin{lemma}\label{openclosed}  If $H\leq \pi_1(X, x)$ is open then $H$ is also closed.
\end{lemma}

\begin{proof}  Supposing $H$ is open we have by Lemma \ref{relations} that the set $\bigcup_{l\notin \bigcup H} [l]H$ is a union of open sets in $L_x$, and this is precisely $L_x \setminus (\bigcup H)$.
\end{proof}

We notice that change of basepoint isomorphisms take open (resp. closed) subgroups to open (resp. closed) subgroups, as seen in the following lemma.

\begin{lemma}\label{pointindependence}  Let $x,y\in X$ and $\rho$ a path from $y$ to $x$.  Let $\phi: L_x\rightarrow L_y$ be the map such that $\phi(l) = \rho*l*\rho^{-1}$ and $\psi: L_y\rightarrow L_x$ be given by $\rho^{-1}*l*\rho$.  Then

\begin{enumerate} \item $\phi$ and $\psi$ are isometric embeddings and induce isomorphisms $\overline{\phi}: \pi_1(X, x) \rightarrow \pi_1(X, y)$, and $\overline{\psi}: \pi_1(X, y) \rightarrow \pi_1(X,x)$.

\item $G\leq \pi_1(X, x)$ is open (resp. closed) iff $\overline{\phi}(G)$ is.

\item $G\leq \pi_1(X, x)$ is open (resp. closed) iff every conjugate of $G$ is.

\end{enumerate}
\end{lemma}
\begin{proof}  The first part of (1) is clear, and the second is a standard exercise.  For (2) suppose $G$ is not open.  Let $l\in L_x$ with $[l]\in G$ and $\{l_n\}_{n\in \omega}$ be a sequence of loops such that $l_n \rightarrow l$ and $[l_n] \notin G$.  Then $\rho*l_n*\rho^{-1} \rightarrow \rho*l*\rho^{-1}$ and $[\rho*l_n*\rho^{-1}] \notin \overline{\phi}(G)$, so $\overline{\phi}(G)$ is not open.  If $\overline{\phi}(G)$ is not open then by the proof for the other direction we have that $\overline{\psi} \overline{\phi}(G) = G$ is not open.

Suppose that $G$ is not closed and let $l\in L_x$ be such that $[l]\notin G$ and there exists a sequence $\{l_n\}_{n\in \omega}$ such that $[l_n] \in G$ and $l_n \rightarrow l$.  Then $\rho*l_n*\rho^{-1} \rightarrow \rho*l*\rho^{-1}$ and $[\rho*l_n*\rho^{-1}]\in \phi(G)$ and $[\rho*l*\rho^{-1}] \notin \phi(G)$.  Again, for the other direction we consider the application of the map $\overline{\psi}$.

The last claim is proved by letting $\rho$ be a loop from $x$ to itself and applying (2).

\end{proof}

By Lemma \ref{pointindependence} we may consider open or closed normal subgroups as base point free.

\begin{lemma}\label{opencoverforopen}  If $G \unlhd \pi_1(X)$ is open there exists an open cover $\mathcal{U}$ of $X$ such that any loop contained entirely in an element of $\mathcal{U}$ is in $\bigcup G$.
\end{lemma}

\begin{proof}  For each point $x\in X$ we have $G\unlhd \pi_1(X, x)$ is open, and the constant loop $c$ at $x$ is in $\bigcup G$, so we may pick $\epsilon_x >0$ such that $B(c, \epsilon_x) \subseteq \bigcup G$.  Selecting the $\epsilon_x$ neighborhood $B(x, \epsilon_x)$ around $x$ gives the desired open cover $\mathcal{U} = \{B(x, \epsilon_x)\}_{x\in X}$.
\end{proof}

The converse to the above lemma is not true in general.  The space $F$ in the next example will reappear in later examples in this paper.

\begin{example}\label{exampleF}  Let $F =  \bigcup_{y \in K} C((0, y),y) \subseteq \mathbb{R}^2$ where $K$ is a homeomorph of the Cantor set that lies in the interval $[1,2]$ and $C(p, r)$ denotes the circle centered at point $p$ of radius $r$.  The space $F$ can be considered a wedge of $2^{\aleph_0}$ many circles of diameter $\geq 1$ whose antipodes from the wedge point $(0,0)$ correspond to the elements of a Cantor set.  This space is compact and the fundamental group is easily seen to be isomorphic to the free group of rank continuum $F(2^{\aleph_0})$ (with a free generating set corresponding to a set of loops that go exactly once around one of the circles $C((0, y),y)$ ).  Let $\Po$ be any pointset defined on metric spaces which is closed under taking continuous preimages.  Define a map $f:K \rightarrow L_{(0,0)}$ by letting $f(y)(t) = (y\sin(2\pi t), y-y\cos(2\pi t))$.  It is clear for $y_0, y_1\in K$ that $d(f(y_0), f(y_1)) = 2d(y_0, y_1)$ since $d(f(y_0)(s), f(y_1)(s))$ is maximized precisely at $s=\frac{1}{2}$ and $d(f(y_0)(\frac{1}{2}), f(y_1)(\frac{1}{2})) = 2d(y_0, y_1)$.  Then $f$ is an embedding of $K$.  The image $f(K)$ gives a set of loops which freely generate the fundamental group.  If $G \leq \pi_1(F, (0,0))$ is of pointclass $\Po$ then $f^{-1}(\bigcup G)$ is as well.

For any $\emptyset \neq S\subseteq K$ we have a subgroup: $$\iota_*(\pi_1( \bigcup_{y \in S} C((0, y),y), (0,0))) \leq \pi_1(F, (0,0))$$ freely generated by the loops in $f(S)$.  The normal closure $$G = \langle\langle \iota_*(\pi_1( \bigcup_{y \in S} C((0, y),y), (0,0)))\rangle\rangle \leq \pi_1(F, (0,0))$$ does not contain any elements of form $[f(y)]$ where $y\in K\setminus S$ since $G$ is the kernel of the retraction map from $\pi_1(F, (0,0))$ to the free subgroup $$\iota_*(\pi_1( \bigcup_{y \in K \setminus S} C((0, y),y), (0,0))) \leq \pi_1(F, (0,0))$$Any loop in $F$ contained in an open ball of radius $\frac{1}{2}$ is nulhomotopic, so there exists an open cover $\mathcal{U}$ satisfying the conclusion of Lemma \ref{opencoverforopen} for any normal subgroup $G\unlhd \pi_1(F, (0,0))$.  Not every subgroup is open, however, by letting $S \subseteq K$ be not open and noticing that $S = f^{-1}(\bigcup\langle\langle \iota_*(\pi_1( \bigcup_{y \in S} C((0, y),y), (0,0)))\rangle\rangle)$ is not open.
\end{example}

 We present a partial converse to Lemma \ref{opencoverforopen}.

\begin{definition}
A topological space $Z$ is \textbf{locally path connected} if for every $z\in Z$ and neighborhood $U$ of $z$ there exists a neighborhood $V \subseteq U$ of $z$ such that $V$ is path connected.
\end{definition}

\begin{lemma}\label{opencover}  Let $X$ be locally path connected and $G \unlhd \pi_1(X)$.  If there exists an open cover $\mathcal{U}$ of $X$ such that any loop contained entirely in an element of $\mathcal{U}$ is in $G$ then $G$ is open.
\end{lemma}

\begin{proof}  Assume the hypotheses and fix $x\in X$.  Let $l\in \bigcup G \subseteq L_x$. Cover the image of $l$ with a finite subcollection $\{U_0, \ldots, U_{m}\} \subseteq \mathcal{U}$, so that the images of each inclusion $\iota_*:\pi_1(U_i) \rightarrow \pi_1(X)$ are in $G$.  Let $\delta>0$ be a Lebesgue number for the covering of the image of $l$ by $\{U_0, \ldots, U_{m}\}$.  Cover $l$ with finitely many open balls $\{B_0, \ldots, B_{k}\}$ of radius $\frac{\delta}{2}$.  Cover the image of $l$ with finitely many path connected open sets $\{V_0, \ldots, V_q\}$, each of which is contained in one of the $\{B_0, \ldots, B_{k}\}$.  Let $\epsilon$ be a Lebesgue number for the covering $\{V_0, \ldots, V_q\}$ of the image of $l$.  Pick $N\in \omega$ sufficiently large so that for $0 \leq n \leq N-1$ we have that $l([\frac{n}{N}, \frac{n+1}{N}])$ is contained inside some $V_{j_n}$.  Now assuming $l'\in L_x$ is less than distance $\epsilon$ from $l$ we have that $d(l'(s), l(s)) < \epsilon$ for all $s\in [0,1]$.  For each $1\leq n \leq N-1$ let $p_n$ be a path in $V_{j_n}$ from $l(\frac{n}{N})$ to $l'(\frac{n}{N})$ and let $p_0$ and $p_N$ be the constant path at $x$.  Notice that the loop $l|[\frac{n}{N}, \frac{n+1}{N}] *p_{n+1} * (l'|[\frac{n}{N}, \frac{n+1}{N}])^{-1}* p_n^{-1}$ is contained in one of the $U_i$, and so is a representative of an element of $G$ based potentially at a different point.  Then $l^{-1}\ast l'$ is an element of $\bigcup G$, so $l' \in \bigcup G$. Thus $G$ is open.
\end{proof}

For the next proposition we recall the following definition.

\begin{definition}
A topological space $Z$ is \textbf{semi-locally simply connected} if for every $z\in Z$ there exists a neighborhood $U$ of $z$ such that the map induced by inclusion $\iota_*: \pi_1(U, z) \rightarrow \pi_1(Z, z)$ is the trivial map.  For a locally path connected space we may obviously select $U$ to be path connected.
\end{definition}

\begin{proposition}\label{slsc}  Let $X$ be locally path connected in addition to being metrizable.  The following are equivalent:

\begin{enumerate}\item The trivial subgroup of $\pi_1(X)$ is open.

\item All subgroups of $\pi_1(X, x)$ are open.

\item $X$ is semi-locally simply connected.

\end{enumerate}

\end{proposition}

\begin{proof}  The implication (1) $\Rightarrow$ (2) follows from Lemma \ref{supergroup}.  For (2) $\Rightarrow$ (3) we let $x\in X$ be given along with a neighborhood $U$ of $x$.  Since in particular the trivial subgroup of $\pi_1(X, x)$ is open and the constant map $c:[0,1] \rightarrow \{x\}$ is trivial, we may select $\epsilon>0$ such that $B(c, \epsilon) \subseteq \bigcup[c] \subseteq L_x$, where without loss of generality $B(x, \epsilon) \subseteq U$.  Now any loop with image in $B(x, \epsilon)$ must be in $B(c, \epsilon)$ and therefore nulhomotopic in $X$.

For (3) $\Rightarrow$ (1) we let $\mathcal{U}$ be an open cover of $X$ by path connected open sets $U$ whose inclusion maps induce a trivial map $\pi_1(U) \rightarrow \pi_1(X)$.  Then we are in the situation of Lemma \ref{opencover} and we see that the trivial subgroup is open, so we are done.
\end{proof}

\end{section}

\begin{section}{Polish Spaces}\label{Polishspaces}

We present some material which will be specific to dealing with fundamental groups of Polish spaces.  We give some technical lemmas which will establish closure properties for subgroups of particularly nice types of pointclasses (stated in Theorem \ref{closureprop}).  These will give a sense of the versatility of such subgroups.  We will then prove a couple of the main results of the paper.  Recall the following:

\begin{definition}  A topological space $Z$ is \textbf{Polish} if it is completely metrizable and separable.
\end{definition}

Many commonly used spaces such as the real line $\mathbb{R}$, compact metric spaces, and countable discrete spaces are Polish.  Polish spaces are closed under countable disjoint union and countable products.  When $X$ is path connected and Polish the space $L_x$ is also Polish.  The space $H_x$ of homotopies of loops at $x$, topologized by the $\sup$ metric, is Polish assuming $X$ is path connected Polish.  The following lemma provides a sense of base point independence as in Lemma \ref{pointindependence}.

\begin{lemma}\label{pathiso}  Suppose the pointclass $\Po$ contains the closed sets and is closed under continuous images between Polish spaces, finite products, and finite intersections. Let $X$ be Polish and $\rho$ be a path from $x$ to $y$ in $X$.  Letting $\phi$ be the map defined in Lemma \ref{pointindependence}, a subgroup $G\leq \pi_1(X, x)$ is of type $\Po$ if and only if $\phi(G)$ is.
\end{lemma}
\begin{proof}  Assume the hypotheses.  We prove the forward direction of the biconditional and the other direction follows similarly.  Let $G\leq \pi_1(X, x)$ be of type $\Po$.  Let $D \subseteq L_x\times H_x \times L_x$ be defined by $D = \{(l_0, H, l_1): H$ is a homotopy from $l_0$ to $l_1\}$.  It is easy to see that $D$ is closed.  Since the map $l \mapsto \rho^{-1}*l*\rho$ is an isometric embedding from $L_x$ to $L_y$ we have that $\rho^{-1}*G*\rho$ is in pointclass $\Po$ in $L_y$ by assumption.  Then $(\rho^{-1}*G*\rho) \times H_y \times L_y$ is in pointclass $\Po$ in $L_y \times H_y \times L_y$ by hypothesis.  Then $D \cap (\rho^{-1}*G*\rho) \times H_y \times L_y$ is in pointclass $\Po$.  Letting $p_3:L_x\times H_x \times L_x\rightarrow L_x$ be projection to the third coordinate (obviously a continuous map), we have that $\bigcup\phi(G) = p_3(D \cap (\rho^{-1}*G*\rho) \times H_y \times L_y)$ is in the pointclass $\Po$.
\end{proof}

For $K\subseteq L_x$ let $[K] \subseteq \pi_1(X, x)$ denote the subset of equivalence classes of loops which have representatives in $K$.

\begin{lemma}\label{loopgeneration}  Let $\Po$ and $X$ satisfy the hypotheses of Lemma \ref{pathiso}.  If $K\subseteq L_x$ is $\Po$ then the set $\bigcup[K] \subseteq L_x$ is $\Po$.
\end{lemma}

\begin{proof}  Letting $D = \{(l_0, H, l_1): H$ homotopes $l_0$ to $l_1\}\subseteq L_x\times H_x\times L_x$ we have that $D$ is closed and therefore $\Po$.  The set $K$ is $\Po$ and therefore so is $K \times H_x\times L_x$.  Then $(K \times H_x\times L_x)\cap D$  is $\Po$, and letting $p_3$ be projection in the third coordinate we have $p_3((K \times H_x\times L_x)\cap D) = \bigcup [K]$ is $\Po$.
\end{proof}

\begin{lemma}\label{loopgenerationsubgroup}  Let $\Po$ and $X$ satisfy the hypotheses of Lemma \ref{pathiso}.  Assume further that $\Po$ is closed under countable unions.  If $K \subseteq L_x$ is $\Po$ then $\langle[K]\rangle$ is a $\Po$ subgroup of $\pi_1(X, x)$.
\end{lemma}

\begin{proof}  Notice that the inversion map $l \mapsto l^{-1}$ is an isometry and therefore continuous.  Thus $K^{-1}$ is $\Po$, and $K \cup K^{-1}$ is also $\Po$.  For each $n\in \omega$ let $m_n: \prod_{i=0}^{n-1} L_x \rightarrow L_x$ be given by $(l_0, \ldots, l_{n-1}) \mapsto l_0*l_1*\cdots *l_{n-1}$.  This is clearly a continuous map.  Each $m_n(\prod_{i=0}^{n-1} (K\cup K^{-1}) )$ is of type $\Po$.  Thus $\bigcup_{n=0}^{\infty}m_n(\prod_{i=0}^{n-1} (K\cup K^{-1}) )$ is $\Po$.  By Lemma \ref{loopgeneration} we have that $\bigcup[\bigcup_{n=0}^{\infty}m_n(\prod_{i=0}^{n-1} (K\cup K^{-1}) )]$is $\Po$.  We are done since $\bigcup \langle[K]\rangle = \bigcup[\bigcup_{n=0}^{\infty}m_n(\prod_{i=0}^{n-1} (K\cup K^{-1}) )]$.
\end{proof}

\begin{lemma}\label{normalclosure} Let $\Po$ and $X$ satisfy the hypotheses of Lemma \ref{loopgenerationsubgroup}.  If $K \subseteq L_x$ is $\Po$ then the normal closure $\langle\langle [K]\rangle\rangle$ is $\Po$.
\end{lemma}

\begin{proof}  Let $c: L_x \times L_x \rightarrow L_x$ be given by $(l_0, l_1) \mapsto l_0*l_1*l_0^{-1}$.  This is easily continuous.  We have $L_x \times K$ is $\Po$, and so is $c(L_x \times K)$.  Then $\langle \langle [K]\rangle\rangle = \langle [c(L_x \times K)] \rangle$ is $\Po$ by Lemma \ref{loopgenerationsubgroup} .
\end{proof}

The preceeding lemmas motivate the following:

\begin{definition}  A pointclass $\Po$ defined on Polish spaces is \textbf{nice} if it contains the closed sets, is closed under continuous images and preimages, countable intersections and finite unions.
\end{definition}

\begin{remark}  A nice pointclass is also closed under countable products, for if $A_n \subseteq Z_n$ is of nice pointclass $\Po$ for each $n\in \omega$ then $\prod_{n\in \omega} A_n = \bigcap_{n\in \omega}p_n^{-1}(A_n)$ is $\Po$ in the Polish space $\prod_{n\in \omega} Z_n$.  A nice pointclass is also closed under countable unions, for suppose $A_n \subseteq Z$ are $\Po$ for each $n\in \omega$.  If $\bigcup_{n\in \omega} A_n = \emptyset$ then as $\emptyset$ is closed we have $\bigcup_{n\in \omega} A_n$ is $\Po$.  Otherwise pick $z\in\bigcup_{n\in \omega} A_n$.  Since $\{z\}$ is closed in $Z$ and $\Po$ is closed under finite unions we can assume $z\in A_n$ for all $n\in \omega$.   Let $\bigsqcup_{n\in \omega}Z$ be the disjoint union of countably many copies of $Z$.  Let $f:\bigsqcup_{n\in \omega} Z \rightarrow \prod_{n\in \omega}Z$ take $y_n$ to $(z, z, \ldots, z, y, z, z\ldots)$ (here $y$ is in the $n$th coordinate) where $y_n$ is a copy of $y$ in the $n$th copy of $Z$ in the disjoint union.  The map $f$ is continuous by the universal and co-universal properties of product and disjoint unions, respectively.  The set $\prod_{n\in \omega} A_n$ is $\Po$ as we have seen.  Letting $g:\bigsqcup_{n\in \omega} Z \rightarrow Z$ map each copy of $Z$ via identity we get that $g(f^{-1}(\prod_{n\in \omega} A_n)) = \bigcup_{n\in \omega} A_n$ is $\Po$.

\end{remark}

Under set inclusion the smallest nice Polish pointclass is that of the analytic sets (denoted $\Sigma_1^1$).  If $Z$ is Polish we say $Y\subseteq Z$ is \textbf{analytic} if there exists a Polish space $W$ and a continuous map $f: W\rightarrow Z$ such that $f(W) = Y$.  All Borel sets of a Polish space are analytic (see \cite{Ke}).

\begin{lemma}\label{productstructure}  If $X = \prod_{n\in\omega} X_n$ where each $X_n$ is metrizable, then the loop space of $X$ is homeomorphic to the product of the loop spaces of the spaces $X_n$ and can be metrized thereby.
\end{lemma}

\begin{proof}  By applying a cutoff metric $d_n$ to each space $X_n$ we may assume $\diam(X_n) \leq 2^{-n}$.  The metric $d(\{s_n\}_{n\in\omega}, \{t_n\}_{n\in\omega}) = \sum_{n=0}^{\infty} d_n(s_n, t_n)$ is compatible with the product topology on $\prod_n X_n$.  Fix a point $x_n$ in each $X_n$ and let $x = \{x_n\}_{n\in \omega} \in \prod_{n} X_n$.  The metric $d$ induces the $\sup$ metric on the loop space $L_x$ so that $L_x$ is homeomorphic with the space $\prod_{n}L_{x_n}$ where the distance between loops $\{l_n\}_{n\in \omega}$ and $\{l_n'\}_{\omega}$ is $ \sum_n \sup_{s\in [0,1]} d_n(l_n(s), l_n'(s))$.  This follows from the fact that uniform convergence of a sequence of loops in $L_x$ occurs precisely when the loops in each coordinate converge uniformly.  Thus we may metrize $L_x$ with the metric defined by the metric on the product $\prod_{n}L_{x_n}$.  

\end{proof}

We cover some functoriality properties.  Recall that if $(X, x)$ and $(Y, y)$ are two pointed spaces and $f:(X, x) \rightarrow (Y, y)$ is a continuous function there is an induced homomorphism $f_* :\pi_1(X, x) \rightarrow \pi_1(Y, y)$ defined by $f_*([l]) = [f\circ l]$.  The map $f$ also induces a continuous map $\overline{f}: L_x \rightarrow L_y$ given by $l\mapsto f\circ l$.  We also recall that the wedge $(X, x) \vee (Y, y)$ is the topological space obtained by identifying the distinguished points, which has distinguished point corresponding to the identified points which we denote $x\vee y$.  There are obvious inclusion maps from the spaces $(X, x)$ and $(Y, y)$ to the wedge as well as retraction maps from the wedge to the two spaces.  If $X$ and $Y$ are metrizable, competely metrizable, or separable then so is the wedge.

\begin{proposition}\label{funct}  Assume $X$ and $Y$ are metric spaces.  The following closure properties hold:

\begin{enumerate}

\item If $f:(X, x) \rightarrow (Y, y)$ is continuous, $\Po$ is a pointclass closed under continuous preimages and $G \leq \pi_1(X, x)$ is $\Po$, then $(f_*)^{-1}(G)$ is also $\Po$. 

\item If $G_0\leq \pi_1(X, x)$ and $G_1 \leq \pi_1(Y, y)$ are both of pointclass $\Po$ and $\Po$ is closed under products, then $G_0\times G_1 \leq \pi_1(X \times Y, (x, y)) \simeq \pi_1(X, x) \times \pi_1(Y,y)$ is $\Po$.

\item If $f:(X, x) \rightarrow (Y, y)$ is continuous between Polish spaces and $\Po$ is nice and $G\leq \pi_1(X, x)$ is $\Po$ then $f_*(G)$ is $\Po$.

\item If $G_0 \leq \pi_1(X, x)$ and $G_0 \leq \pi_1(Y, y)$ are $\Po$, with $X$ and $Y$ Polish and $\Po$ nice, then the subgroup generated by the images of $G_0$ and $G_1$ under the inclusion maps is $\Po$ in $\pi_1((X, x) \vee (Y, y), x\vee y)$.
\end{enumerate}

\end{proposition}

\begin{proof}  For part (1) we notice that $\bigcup (f_*)^{-1}(G) = \overline{f}^{-1}(\bigcup G)$.  Claim (2) follows from Lemma \ref{productstructure}, and applies to countable products if $\Po$ is closed under countable products.  For (3) the map $f$ induces the continuous map $\overline{f}$ from $L_x$ to $L_y$ by composition.  The image of $\bigcup G$ under this map is $\Po$ because $\Po$ is nice, and $\bigcup f_*(G) = [\overline{f}(\bigcup G)]$.  Claim (4)  follows immediately, since $\bigcup\langle\iota_{X*}(G_0) \cup \iota_{Y*}(G_1)\rangle = \bigcup\langle[\iota_X(\bigcup G_0)\cup \iota_Y(\bigcup G_1)]\rangle$ is evidently $\Po$.
\end{proof}

The following theorem gives a catalogue of closure properties for nice subgroups.  Recall that the derived series is defined by letting $G^{(0)} = G$, $G^{(\alpha+1)} = [G^{(\alpha)}, G^{(\alpha)}]$ and $G^{(\beta)} = \bigcap_{\alpha<\beta}G^{(\alpha)}$ if $\beta$ is a limit ordinal.  The lower central series is defined by letting $G_0 = G$ and $G_{n+1} = [G, G_n]$.

\begin{theorem} \label{closureprop} Let $f:(X, x) \rightarrow (Y, y)$ be a continuous function between Polish spaces and let $\Po$ be a nice pointclass.  The following hold:

\begin{enumerate}

\item If $H\leq  \pi_1(Y, y)$ is $\Po$ then $f_*^{-1}(H)\leq \pi_1(X, x)$ is $\Po$.

\item If $G\leq \pi_1(X, x)$ is $\Po$ then $f_*(G)\leq \pi_1(Y, y)$ is $\Po$.

\item The subgroups $1$ and $\pi_1(X, x)$ are analytic in $\pi_1(X, x)$.

\item If $G_n \leq \pi_1(X, x)$ are $\Po$ then so are $\displaystyle\bigcap_{n\in \omega} G_n$ and $\displaystyle\langle \bigcup_{n\in \omega} G_n \rangle$.

\item Countable subgroups of $\pi_1(X, x)$ are analytic.

\item If $G\leq \pi_1(X, x)$ is $\Po$ then so is $\langle \langle G \rangle\rangle$.

\item If $G\leq \pi_1(X,x)$ is $\Po$ then so is any conjugate of $G$.

\item If $w(x_0, \ldots, x_k)$ is a reduced word in the free group $F(x_0, \ldots, x_k)$ and the groups $G_0, \ldots, G_k \leq \pi_1(X, x)$ are $\Po$ then so is the subgroup $\langle \{w(g_0, g_1, \ldots, g_k)\}_{g_i \in G_i}\rangle$.

\item If $G, H \leq \pi_1(X, x)$ are $\Po$ then so is the subgroup $[G, H]$.

\item If $G\leq \pi_1(X, x)$ is $\Po$ then each countable index term of the derived series $G^{(\alpha)}$ and each term of the lower central series $G_n$ is $\Po$.
\end{enumerate}
\end{theorem}

\begin{proof} Claim (1) follows from (1) in Proposition \ref{funct}.  Claim (2) is claim (3) in Proposition \ref{funct}.  For (3) we have that $\pi_1(X, x)$ is a closed subgroup and $1$ is the subgroup generated by the constant map to $x$, and so is analytic by Lemma \ref{loopgenerationsubgroup} (since a singleton is closed in $L_x$).  Claim (4) follows from the definition of nice pointclasses and Lemma \ref{loopgenerationsubgroup}.  Claim (5) follows from the fact that singletons are closed in $L_x$ and claim (4).  Claim (6) is an instance of Lemma \ref{normalclosure}.  Claim (7) is an instance of Lemma \ref{pathiso}.  For claim (8) we notice that the map $w: \prod_{i=0}^k L_x \rightarrow L_x$ given by $(l_0, \ldots, l_k) \mapsto w(l_0, \ldots, l_k)$ is continuous, and so $\{w(l_0, \ldots, l_k)\}_{l_i \in \bigcup G_i}$ is a $\Po$ subset in $L_x$ and the claim follows from Lemma \ref{loopgenerationsubgroup}.  Claim (9) is an instance of claim (8).  For claim (10) we iterate claim (9), applying claim (4) at limit ordinals.
\end{proof}

We recall some definitions.  If $Z$ is a topological space we say that $Y\subseteq Z$ is \textbf{nowhere dense} if the closure $\overline{Y} \subseteq Z$ has empty interior, $Y$ is \textbf{meager} if it is a union of countably many nowhere dense sets in $Z$, $Y$ has the \textbf{property of Baire} (abbreviated BP) if there exists an open set $O \subseteq Z$ such that $Y\Delta O = (Y\setminus O) \cup (O\setminus Y)$ is meager, and $Y$ is \textbf{comeager} if $Z\setminus Y$ is meager.  We say a pointclass $\Po$ on Polish spaces has the property of Baire if each set in $\Po$ has the property of Baire.  For example, the pointclass of open sets obviously has BP.  In fact, the class of analytic sets also has BP (see \cite{Ke}).

The following was proven in \cite{P} using a result from \cite{My}.

\begin{lemma*}\label{Pawlikowski}  Suppose $\approx$ is an equivalence relation on the Cantor set $\{0,1\}^{\omega}$ such that if $\alpha$ and $\beta$ differ at exactly one coordinate then $\alpha\approx \beta$ fails.  If $\approx$ has BP as a subset of $\{0,1\}^{\omega} \times \{0,1\}^{\omega}$, then $\approx$ has $2^{\aleph_0}$ equivalence classes.
\end{lemma*}

\begin{lemma}\label{uncountable}  Let $X$ be Polish.  Suppose that $G \unlhd K  \leq \pi_1(X, x)$ with $G$ of pointclass $\Po$ and that $K$ is closed.  Suppose also that $\Po$ has BP and is closed under continuous preimages in Polish spaces, and that there exist arbitrarily small loops at $x$ which are in $\bigcup K$ and not in $\bigcup G$.  Then $\card(K/G) = 2^{\aleph_0}$.
\end{lemma}

\begin{proof}  Assume the hypotheses and let $\{l_n\}_{n\in \omega}$ be a sequence of loops at $x$ in $\bigcup (K \setminus G)$ such that the diameter of $l_n$ is $\leq 2^{-n}$.  Let $l_n^{0}$ be the constant loop at $x$ and let $l_n^1$ be the loop $l_n$.  Given an element $\alpha \in \{0, 1\}^{\omega}$ we define $l^{\alpha}$ to be the loop $l_0^{\alpha(0)}*(l_1^{\alpha(1)}*(l_2^{\alpha(2)}*(\cdots )))$ (which must also be in $\bigcup K$ as $K$ is closed).  In other words, $l^{\alpha}$ restricted to the interval $[0,\frac{1}{2}]$ is either the constant loop or $l_0$ in case $\alpha(0)$ is $0$ or $1$ respectively, $l^{\alpha}$ restricted to the interval $[\frac{1}{2}, \frac{3}{4}]$  is either the constant loop or $l_1$ in case $\alpha(1)$ is $0$ or $1$ respectively, etc.  The function from the Cantor set $\{0, 1\}^{\omega}$ to $L_x$ given by $\alpha \mapsto l^{\alpha}$ is clearly continuous.  For $l, l' \in L_x$ letting $l \sim l'$ if and only if $[l]G = [l']G$, we have by Lemma \ref{relations} that $\sim \subseteq L_x\times L_x$ is of pointclass $\Po$.  Defining an equivalence relation $\approx$ on $\{0,1\}^{\omega}$ so that $\alpha \approx \beta$ if and only if $l^{\alpha} \sim l^{\beta}$, we see that $\approx \subseteq \{0, 1\}^{\omega} \times \{0,1\}^{\omega}$ is of pointclass $\Po$ as a continuous preimage.  As $\Po$ has BP we know that $\approx$ has BP.  By Lemma \ref{Pawlikowski} we shall be done if we show that if $\alpha$ and $\beta$ differ at exactly one point then $\alpha \approx \beta$ fails.  Suppose that $\alpha(n) \neq \beta(n)$ and that $\alpha(m) = \beta(m)$ whenever $m\neq n$ and that $l^{\alpha}\approx l^{\beta}$.  Letting without loss of generality $\alpha(n) = 1$ and $\beta(n) = 0$ we see that $[(l^{\beta})^{-1}*l^{\alpha}] \in G$.  Let $h = l_{n+1}^{\alpha(n+1)}*(l_{n+2}^{\alpha(n+2)}*(\cdots ))$ and $g = l_0^{\alpha(0)}*(l_1^{\alpha(1)}*(\cdots l_{n-1}^{\alpha(n-1)})\cdots)$.  Then $[(l^{\beta})^{-1}*l^{\alpha}] = [h^{-1}*g^{-1}*g*l_n*h] = [h^{-1}* l_n * h]\in G$, so by normality of $G$ in $K$ we have $[l_n] \in G$, a contradiction.  Thus there are at least $2^{\aleph_0}$ many elements in $K/G$ by the above lemma, and there are at most $2^{\aleph_0}$ elements because there are at most $2^{\aleph_0}$ loops at $x$.
\end{proof}

\begin{lemma}\label{openorsequence} If $X$ is metric, locally path connected and $G \unlhd \pi_1(X)$ then either $G$ is open or there exists $y\in X$ and a sequence of loops $\{l_n\}_{n\in \omega}$ at $y$ with $\diam(l_n) \searrow 0$ and $[l_n]\notin G$.
\end{lemma}

\begin{proof}  If $G$ is not open we have by the contrapositive of Lemma \ref{opencover} that there must exist some point $y \in X$ such that for any open neighborhood $U$ of $y$ there is a loop in $U$ which is not in $G$.  We get a sequence of loops $\{l_n\}_{n\in \omega}$ with $\diam(l_n([0,1]) \cup \{y\}) \leq 2^{-n}$ and $[l_n]\notin G$.  By local path connectedness we may take a subsequence of the $l_n$ whose base point is close enough to $y$ and join the basepoint to $y$ via a small path, so that the loops may eventually be assumed to have been based at $y$.  Taking a subsequence having all loops based at $y$ gives the desired result
\end{proof}

Through the remainder of Section \ref{Polishspaces} we shall assume $\Po$ is a nice pointclass with BP.  Thus for example one can take $\Po = \Sigma_1^1$.  The discussion of pointclasses will take place in Section \ref{conclusion}.

\begin{theorem}\label{unc}  Suppose $X$ is locally path connected Polish.  If $G \unlhd \pi_1(X)$ is $\Po$ then $\card(\pi_1(X)/G)$ is either $\leq \aleph_0$ (in case $G$ is open) or  $2^{\aleph_0}$ (in case $G$ is not open).
\end{theorem}

\begin{proof}  If $G$ is open then the collection of left cosets $\{[l]G\}_{l\in  L_x}$ is a covering of $L_x$ by pairwise disjoint open sets, and since $L_x$ is separable we know that the collection $\{[l]G\}_{l\in  L_x}$ is countable.  Else, by Lemma \ref{openorsequence} we get a point $y\in X$ and a sequence of loops $\{l_n\}_{n\in \omega}$ with $\diam(l_n)\searrow 0$ and $[l_n]\notin G$.  Considering $G$ as a subgroup of $\pi_1(X, y)$ we see that $G$ is $\Po$ since $\Po$ is nice, and thus we have satisfied the hypotheses of Lemma \ref{uncountable} and we are done.
\end{proof}

The above may be strengthened if $X$ is also compact.  Recall that a \textbf{Peano continuum} is a path connected, locally path connected compact metrizable space.

\begin{theorem}\label{Shelahgen}  If $X$ is a Peano continuum and $G \unlhd \pi_1(X, x)$ is $\Po$ then $\pi_1(X, x)/G$ is either finitely generated (in case $G$ is open) or of cardinality $2^{\aleph_0}$ (in case $G$ is not open).
\end{theorem}

\begin{proof}  By Theorem \ref{unc} we need only show that $\pi_1(X)/G$ is finitely generated if $G$ is open.  For this we will use a theorem from \cite{CC} which will require a definition.  Let $\phi: \pi_1(X) \rightarrow H$ be a group homomorphism.  We say an open cover $\mathcal{U}$ is \textbf{2-set simple rel $\phi$} if each element of $\mathcal{U}$ is path connected and any loop in the union of two elements of $\mathcal{U}$ is in the kernel of $\phi$.  This property of a cover implies that for any nerve associated with with $\mathcal{U}$ there is a homomorphism from the fundamental group of the nerve with the same image as $\phi$.  The following is a part of Theorem 7.3 in \cite{CC}:
\begin{theorem*}
Let $X$ be path connected, $\phi: \pi_1(X) \rightarrow H$ a homomorphism and $\mathcal{U}$ a $2$-set simple cover rel $\phi$.  If $\mathcal{U}$ is finite then $\phi(\pi_1(X))$ is finitely generated.
\end{theorem*}

Now, assuming $G$ is open we get by Lemma \ref{opencoverforopen} an open cover $\mathcal{U}_1$ for $X$ such that any loop contained in an element of $\mathcal{U}_1$ is in $G$.  Let $\epsilon>0$ be a Lebesgue number for the cover $\mathcal{U}_1$ and let $\mathcal{U}_2$ be a cover of $X$ by open balls of redius $\frac{\epsilon}{4}$.  By local path connectedness let $\mathcal{U}$ be an open cover of $X$ by path connected sets, each of which is contained in an element of $\mathcal{U}_2$.  By compactness we may pick $\mathcal{U}$ to be finite, and it is clear that $\mathcal{U}$ is $2$-set simple rel the quotient projection $\pi_1(X) \rightarrow \pi_1(X)/G$.  We are done by the theorem of Cannon and Conner that is quoted above.
\end{proof}

The conditions on $G$ and the pointclass $\Po$ cannot be removed in Theorems \ref{unc} and \ref{Shelahgen}  as evidenced by the following example.

\begin{example}\label{productofprojectiveplanes}  Let $P$ denote the projective plane and $ P^{\omega} = \prod_{\omega}P$.  Let $x\in P$ and let $\overline{x} \in P^{\omega}$ be the point whose every coordinate is $x$.  By the functoriality of the fundamental group there is a natural isomorphism $\pi_1(P^{\omega}, \overline{x}) \simeq \prod_{\omega}\pi_1(P, x)$.  By Lemma \ref{productstructure} the loop space $L_{\overline{x}}$ is homeomorphic to the product $\prod_{\omega} L_x$ where the loop space $L_x$ at the $n$th coordinate has diameter $\leq 2^{-n}$.  By the coordinatewise isomorphism $\prod_{\omega}\pi_1(P, x) \simeq \prod_{\omega}\mathbb{Z}/2\mathbb{Z}$ we may regard elements of $\pi_1(P^{\omega}, \overline{x})$ as $\omega$ sequences of $0$s and $1$s.  By the shrinking of the metrics on the coordinate loop spaces, any $g\in \pi_1(P^{\omega}, \overline{x})$ whose first $n$ coordinates are $0$s has a representative loop $l\in g$ with $\diam(l) \leq 2^{-n}$.  The space $P^{\omega}$ is a Peano continuum.

Let $\kappa$ be any cardinal satisfying $\aleph_0 \leq \kappa \leq 2^{\aleph_0}$.  As $\prod_{\omega} \mathbb{Z}/2\mathbb{Z}$ is a $\mathbb{Z}/2\mathbb{Z}$ vector space, we select a linearly independent set $W\subseteq \prod_{\omega} \mathbb{Z}/2\mathbb{Z}$ such that $W \supseteq \{(1, 0, 0,0, \ldots), (0,1, 0, 0,\ldots), (0,0,1, 0, \ldots), \ldots)\}$ and $\card(W) = \kappa$.    Let $B$ be a basis for $\prod_{\omega} \mathbb{Z}/2\mathbb{Z}$ with $B \supseteq W$.  Letting $H = \langle W\rangle$ and $G = \langle W\setminus B\rangle$ we have $\prod_{\omega} \mathbb{Z}/2\mathbb{Z} \simeq H \bigoplus G$.  Now $G\unlhd \pi_1(P^{\omega}, \overline{x})$ and $\pi_1(P^{\omega}, \overline{x})/G \simeq H$ is of cardinality $\kappa$.

Letting $\kappa = \aleph_0$ gives a counterexample to the claim of Lemma \ref{uncountable} with the ``$G$ of pointclass $\Po$'' and ``$\Po$ has BP and is closed under continuous preimages in Polish spaces'' hypotheses removed.  By considering a model of \textbf{ZFC} in which \textbf{CH} fails, we let $\kappa = \aleph_1$ and see that Theorem \ref{unc} can fail if the conditions on $G$ and $\Po$ are removed.  Similarly removing these conditions can make the conclusion of Theorem \ref{Shelahgen} fail.  The quotient can be a countable infinitely generated group or a group of cardinality violating \textbf{CH} if one is in a model of \textbf{ZFC} + $\neg$ \textbf{CH}.
\end{example}

The conclusion of Theorem \ref{Shelahgen} cannot be strengthened by replacing ``finitely generated'' by ``finitely presented'' by the following basic example.

\begin{example}\label{bouquetexample}  Let $X$ be the bouquet of two circles and $H$ be a $2$-generated group which is not finitely presented (for example, the lamplighter group).  The fundamental group $\pi_1(X)$ is the free group of rank $2$.  Let $\phi: \pi_1(X) \rightarrow H$ be a homomorphism given by taking each of the free generators of $\pi_1(X)$ to a distinct generator of $H$.  The space $X$ is a semilocally simply connected Peano continuum and $\ker(\phi) = G$ is open by Proposition \ref{slsc}, but $\pi_1(X)/G \simeq H$ is not finitely presented.  Similar examples can be given by replacing the number $2$ by any finite number $\geq 2$ and letting $G$ be replaced by any other $n$-generated group which is not finitely presented.
\end{example}

We name some of the numerous applications of the above theorems.

\begin{A}  Suppose $X$ is a locally path connected Polish space.  The following groups are of cardinality $2^{\aleph_0}$ or $\leq \aleph_0$, and in case $X$ is compact they are of cardinality $2^{\aleph_0}$ or are finitely generated:

\begin{enumerate}  \item $\pi_1(X)$

\item $\pi_1(X)/(\pi_1(X))^{(\alpha)}$ for any $\alpha<\omega_1$ (derived series)

\item $\pi_1(X)/(\pi_1(X))_n$ for any $n\in \omega$ (lower central series)

\item $\pi_1(X)/N$ where $N$ is the normal subgroup generated by squares of elements, cubes of elements, or $n$-th powers of elements
\end{enumerate}

\noindent In case $X$ is compact then countability of the fundamental group is equivalent to being finitely presented.
\end{A}

\begin{proof}  The noncompact case in parts (1)-(4) immediately follow from Theorem \ref{unc}.  For parts (2)-(4) in the compact case we apply Theorem \ref{Shelahgen}.  That $\pi_1(X)$ would be finitely presented follows from Theorem 7.3 in \cite{CC} in part (1) assuming $X$ is compact.
\end{proof}

Part (1) in the compact case is the main result of the papers \cite{Sh2} and \cite{P}, and part (2) with $\alpha = 1$ (both compact and noncompact cases) is proven in \cite{CoCo}.

Towards proving Theorem \ref{ascendingnormal} we give the following technical lemma.

\begin{lemma}\label{ascending}   Let $X$ be a Polish space.  Suppose $N\unlhd K \leq \pi_1(X, x)$ is such that $\bigcup N = \bigcup_{n = 0}^{\infty} N_n$ with each $N_n$ closed under inverses and homotopy and containing the trivial loop, and that $K$ is closed.  Assume also that $N_n*N_{m}\subseteq N_{n+m}$.  If each $N_n$ is $\Po$ and for each $n\in \omega$ there exist loops at $x$ of arbitrarily small diameter in $\bigcup K$ not contained in $N_n$, then $K/N$ is of cardinality $2^{\aleph_0}$.

\end{lemma}

\begin{proof}  Let $\epsilon>0$ be given.  By the proof of Lemma \ref{uncountable} we need only show that there is a loop at $x$ in $\bigcup K$ of diameter less than $\epsilon$ that is not in $\bigcup N$, since $N$ is $\Po$.  For contradiction we assume that no such loop exists.  For each loop in $\bigcup K$ of diameter less than $\epsilon$ let $\phi$ map that loop to the minimal $k$ such that $l\in N_k$.  For two loops $l_1, l_2$ of radius less than $\epsilon$ we have that $\phi(l_1\ast l_2)\leq \phi(l_1) + \phi(l_2)$ and $\phi(l_1) = \phi(l_1^{-1})$.  Let $\{l_n\}_{n\in \omega}$ be a sequence of loops such that $\diam(l_n) < \epsilon 2^{-n}$ and that $\phi(l_0) >1$ and $\phi(l_n) > n + \sum_{m=0}^{n-1} \phi(l_m)$.  In particular none of the $l_n$ is nullhomotopic.  Define $l^{\alpha}$ as before for each $\alpha$ in the Cantor set.  Abuse notation by letting $\phi:\{0,1\}^{\omega}\rightarrow \omega$ be defined by $\phi(\alpha) = \phi(l^{\alpha})$.

Let $E_n = \{\alpha \in \{0,1\}^{\omega}: l^{\alpha} \in N_n\}$.  As we are assuming that there is no loop in $\bigcup K$ of diameter less than $\epsilon$  that is not in $\bigcup N$, we have $\bigcup_{n=0}^{\infty}E_n = \{0,1\}^{\omega}$.  We will derive our contradiction if we show that each $E_n$ is meager, which would imply that $\{0,1\}^{\omega}$ is meager in itself.  Each $E_n$ is $\Po$, and so has the property of Baire.  Supposing $E_n$ is not meager there exists a nonempty open set in which $E_n$ is comeager.  In particular there exists a basic open set $U(\delta_0, \ldots, \delta_k) = \{\alpha \in \{0,1\}^{\omega}: \alpha(0) = \delta_0,\ldots, \alpha(k) = \delta_k\}$ such that $E_n \cap U(\delta_0, \ldots, \delta_k)$ is comeager in $U(\delta_0, \ldots, \delta_k)$.  For each $p \geq k+1$ let $\beth_p:U(\delta_0, \ldots, \delta_k) \rightarrow U(\delta_0, \ldots, \delta_k)$ be the homeomorphism that changes the $p$ coordinate.  Then $U(\delta_0, \ldots, \delta_k)\setminus \beth_p(E_n)$ is meager for each $p \geq k+1$.  Then in fact there exists $\alpha \in U(\delta_1, \ldots, \delta_k)$ such that switching finitely many of the coordinates beyond the $k$th coordinate gives an element of $E_n$.  It cannot be that the support of $\alpha$ is finite, for if $N\in \omega$ is a bound on the support of $\alpha$ (we can assume $N >2n$), then $n \geq \phi(l^{\alpha}\ast f_{N+1}) \geq \phi(l_{N+1}) - \phi(l^{\alpha})>N+1 - n > n$, a contradiction.  Thus taking a subsequence of the $l_n$, we may assume that $\alpha = (1, 1, \ldots)$ and that $U(\delta_1, \ldots, \delta_k) = \{0, 1\}^{\omega}$.  We assume that this subsequence was the original sequence.

Let $\beta_k, \gamma_k \in \{0,1\}^{\omega}$ be given by $\beta_k(m) = \begin{cases} 0$ if $m<k\\ 1$ if $m\geq k   \end{cases}$ and $\gamma_k(m) = \begin{cases} 1$ if $m<k\\ 0$ if $m \geq k   \end{cases}$.  We have that $\phi(\gamma_k) \geq k$ and $\phi(\beta_k) \geq \phi(\gamma_k) - \phi(\alpha) \geq k-\phi(\alpha)$, so that if $k = 2n+1$ we have on the one hand that $\beta_k \in E_n$ and on the other hand $\phi(\beta_k) \geq k-\phi(\alpha)\geq (2n+1) - n$, a contradiction.
\end{proof}

This gives the following:

\begin{B}  If $X$ is a Peano continuum there does not exist a strictly increasing infinite sequence of $\Po$ normal subgroups $\{G_n\}_{n\in \omega}$ of $\pi_1(X)$ such that $\bigcup_{n\in \omega} G_n = \pi_1(X)$.
\end{B}

\begin{proof}  For each $n\in \omega$ let $\bigcup G_n = N_n$ in the notation of Lemma \ref{ascending}.  If $\pi_1(X)/G_n$ is finitely generated for some $n$, then the sequence $\{N_n\}_{n\in \omega}$ cannot be strictly increasing.  Then $\pi_1(X)/G_n$ must be uncountable for each $n$, so for each $n$ there exist arbitrarily small loops not in $G_n$ by the proof of Theorem \ref{unc}.  By picking an appropriate basepoint by local path connectedness, we are done by Lemma \ref{ascending}.
\end{proof}

The necessity of the conditions on $\Po$ can be easily seen by considering the decomposition $\displaystyle\pi_1(P^{\omega}, \overline{x}) = (\bigoplus_{\omega}\mathbb{Z}/2\mathbb{Z}) \oplus (\bigoplus_{2^{\aleph_0}}\mathbb{Z}/2\mathbb{Z})$ given in Example \ref{productofprojectiveplanes} and letting $\displaystyle G_n = (\bigoplus_{k\leq n}\mathbb{Z}/2\mathbb{Z}) \oplus (\bigoplus_{2^{\aleph_0}}\mathbb{Z}/2\mathbb{Z})$.

\begin{example}  The dual analog of Theorem \ref{ascendingnormal} does not hold: there exists a Peano continuum with an infinite strictly descending chain of analytic (in fact closed) normal subgroups whose intersection is the trivial subgroup.  We again use the space $P^{\omega}$ from Example \ref{productofprojectiveplanes}.  We change the superscript $\omega$ for $\mathbb{Q}$ and the group $\{0,1\}^{\mathbb{Q}}$ remains unchanged since the cardinalities of $\omega$ and $\mathbb{Q}$ are the same.  Given any subset $S\subseteq \mathbb{Q}$ the subgroup of $\pi_1(P^{\omega})$ corresponding to the subgroup $\{\alpha \in \{0,1\}^{\mathbb{Q}}: \alpha(q) = 1 \Rightarrow q\in S\} \leq \{0,1\}^{\mathbb{Q}}$ is closed.  For each $r\in \mathbb{R}$ let $G_r \leq \pi_1(P^{\omega})$ be the subgroup $\{\alpha \in \{0,1\}^{\mathbb{Q}}: \alpha(q) = 1 \Rightarrow q<r\}$.  Then each $G_r$ is a closed subgroup and the following hold:

\begin{enumerate}  \item  $r_n \nearrow r$ implies $\bigcup_{n\in \omega} G_{r_n} < G_r$

\item $r_n \searrow r$ implies $\bigcap_{n\in \omega} G_{r_n}  = G_r$

\item $\bigcap_{r\in \mathbb{R}}  G_r$ is the trivial subgroup
\end{enumerate}

Picking a sequence $r_n\searrow -\infty$ gives a strictly descending sequence of normal analytic subgroups $G_{r_n}$ as claimed.  The subgroup $\bigcup_{r\in \mathbb{R}} G_r$ cannot be equal to $\pi_1(X)$ (else we could pick any sequence $r_n \nearrow \infty$ and the ascending chain $G_{r_n}$ would contradict Theorem \ref{ascendingnormal}).  For example the sequence over $\mathbb{Q}$ which is constantly $1$ is not in $\bigcup_{r\in \mathbb{R}} G_r$.
\end{example}

We now address what can happen in the absence of local path connectedness.  Before stating the next theorem we quote the famous selection theorems of Silver and Burgess (from \cite{Si} and \cite{Bu} respectively).  A set $Y\subseteq Z$ in a Polish space is \textbf{coanalytic} if $Z\setminus Y$ is analytic.  The class of coanalytic sets is denoted $\Pi_1^1$.

\begin{theorem*}(J. Silver)  Suppose $E \subseteq Z\times Z$ is a coanalytic equivalence relation on a Polish space $Z$.  Then either there are $\leq \aleph_0$ many equivalence classes or there exists a homeomorph of a Cantor set $\mathcal{C} \subseteq Z$ such that for distinct $x,y \in \mathcal{C}$ we have $\neg xEy$. 
\end{theorem*}

\begin{theorem*}(J. Burgess)  Suppose $E \subseteq Z\times Z$ is an analytic equivalence relation on a Polish space $Z$.  Then either there are $\leq \aleph_1$ many equivalence classes or there exists a homeomorph of a Cantor set $\mathcal{C} \subseteq Z$ such that for distinct $x,y \in \mathcal{C}$ we have $\neg xEy$. 
\end{theorem*}

\begin{theorem}\label{CH}  Suppose $X$ is path connected Polish and $G \leq \pi_1(X,x)$.

\begin{enumerate}\item  If $G$ is coanalytic then the index $\pi_1(X,x):G$ is either $\leq \aleph_0$ or $2^{\aleph_0}$.

\item If $G$ is analytic then the index $\pi_1(X, x):G$ is either $\leq \aleph_1$ or $2^{\aleph_0}$.

\end{enumerate}

\end{theorem}

\begin{proof}  If $G$ is coanalytic (repesctively analytic) then by Lemma \ref{relations} the equivalence relation induced by the left coset partition is coanalytic (respectively analytic).  Apply the theorem of Silver (resp. Burgess) to conclude (1) (resp. (2)).
\end{proof}

\begin{example}\label{goodexample}  We give an instructive example which demonstrates the sharpness of Theorem \ref{CH} as well as the sharpness of Theorems \ref{unc} and \ref{Shelahgen} in a different sense than that of Example \ref{productofprojectiveplanes}.  There exists a model $\mathcal{M}$ of set theory satisfying

\begin{enumerate}\item \textbf{ZFC}

\item $\neg$ \textbf{CH}

\item Any subset of $\{0,1\}^{\omega}$ of cardinality $\aleph_1$ is coanalytic.

\end{enumerate}

\noindent (see \cite{MaSo}).  We consider the space $F$ and the function $f:K \rightarrow L_{(0,0)}$ used in Example \ref{exampleF} within the model $\mathcal{M}$.  Let $S \subseteq K$ be such that $\card(K \setminus S) = \aleph_1$.  Then $S$ is analytic in $K$.  Then $f(S)$ is analytic in $L_{(0,0)}$.  Then $\langle\langle [f(S)]\rangle\rangle \leq \pi_1(F, (0,0))$ is analytic by Lemma \ref{normalclosure}.  The quotient $\pi_1(F, (0,0))/\langle\langle [f(S)]\rangle\rangle$ is isomorphic to a free group of rank $\aleph_1$ so that $\card(\pi_1(F, (0,0))/\langle\langle [f(S)]\rangle\rangle)= \aleph_1$.

This demonstrates that the case $\aleph_1$ in Theorem \ref{CH} (2) can obtain in the absence of \textbf{CH}.  It also shows that one cannot hope to extend Theorem \ref{CH} part (1) to a higher projective class (since any higher projective class also contains the analytic sets, see Section \ref{conclusion}).  This also shows that one cannot drop local path connectedness in Theorems \ref{unc} and \ref{Shelahgen} and obtain the same conclusion.

\end{example}

\end{section}

\begin{section}{Comonster Groups}\label{comonstergroups}

As an application of the above theory we give the following definition.

\begin{definition}  We say a group $G$ is \textbf{comonster} if for every finite subset $S \subseteq G$ we have $\langle\langle S \rangle\rangle \neq G$.  More generally $G$ is $\kappa$-\textbf{comonster} if for every $S\subseteq G$ with $S$ of cardinality $< \kappa$ we have $\langle\langle S \rangle\rangle \neq G$.
\end{definition}

Thus comonster groups are $\aleph_0$-comonster groups.  One easily sees that any abelian group of cardinality $\kappa> \aleph_0$ is $\kappa$-comonster.  Also, if $h: G\rightarrow H$ is an epimorphism with $H$ comonster (respectively $\kappa$-comonster), then $G$ is also comonster (resp. $\kappa$-comonster).

We assume still that $\Po$ is nice with BP.  We have the following:

\begin{theorem}\label{comonster}  Let $X$ be a Peano continuum and $N\unlhd \pi_1(X)$ of type $\Po$.  If $\pi_1(X)/N$ is comonster then $\pi_1(X)/N$ is $\aleph_1$-comonster.  In particular, if $\pi_1(X)$ is comonster, then $\pi_1(X)$ is $\aleph_1$-comonster.
\end{theorem}

\begin{proof}  Suppose for contradiction that $\pi_1(X)/N$ is comonster but not $\aleph_1$-comonster.  Let $S = \{g_0, \ldots\}\subseteq \pi_1(X)$ be a countably infinite set such that $\langle\langle N\cup S \rangle \rangle= \pi_1(X)$.  The normal groups $G_n = \langle\langle N\cup \{g_0, \ldots, g_n\}\rangle \rangle$ are easily seen to be $\Po$ and $\bigcup_{n} G_n = \pi_1(X)$.  On the other hand the sequence $G_n$ cannot stabilize since $\pi_1(X)/N$ is comonster.  Thus one can pick a strictly increasing subsequence of normal $\Po$ subgroups whose union is $\pi_1(X)$, contradicting Theorem \ref{ascendingnormal}.
\end{proof}

\begin{example}\label{firsthawaiianearring}  Consider the Hawaiian earring $E=\bigcup_{n\in \omega} C((0, \frac{1}{n+2}),\frac{1}{n+2})$.  We have an epimorphism $h:\pi_1(E)\rightarrow \prod_{\omega}\mathbb{Z}$ given by letting the $n$-th coordinate of $h([l])$ count the number of times a loop traverses the $n$-th circle of the infinite wedge that defines $E$ in an oriented direction.  Then $\pi_1(E)$ is $2^{\aleph_0}$-comonster, since $\prod_{\omega}\mathbb{Z}$ is abelian of cardinality $2^{\aleph_0}$.
\end{example}

\begin{example}  If $X$ is a one-dimensional Peano continuum with $\pi_1(X)$ uncountable, then $X$ retracts to a subspace that is homeomorphic to $E$, so that again $\pi_1(X)$ is $2^{\aleph_0}$-comonster.
\end{example}

Even if $\pi_1(X)$ is uncountable, it may still be the case that $\pi_1(X)$ is not comonster, as illustrated in the following example.

\begin{example}  Let $Y$ be a Peano continuum with $\pi_1(Y) \simeq A_5$.  Such a $Y$ exists by taking a finite presentation for $A_5$ and constructing the finite 2-dimensional CW complex by letting loops correspond to generators in the presentation and gluing on the boundary of a disc along a path that gives the relators.  Such a space is compact, metrizable, path connected and locally path connected.  Thus such a $Y$ is a Peano continuum, and so is $X = \prod_{\omega} Y$.  We have $\pi_1(X) \simeq \prod_{\omega} A_5$.  Letting $g \in \prod_{\omega} A_5$ have every entry be the $3$-cycle $(123)$, we claim that $\langle\langle g\rangle\rangle = \prod_{\omega} A_5$.  This demonstrates that $\pi_1(X)$ is not comonster.

To see that $\langle\langle g\rangle\rangle = \prod_{\omega} A_5$, notice that all $3$-cycles are conjugate (in $A_5$) to each other.  Thus for each $h\in \prod_{\omega} A_5$ whose each entry is a $3$-cycle we have $h \in \langle\langle g\rangle\rangle$.  Each $3$-cycle is a product of two $3$ cycles (if $(abc)$ is a $3$-cycle then $(abc) = (abc)^{-1}(abc)^{-1} = (cba)(cba)$).  Since the trivial element in $A_5$ is a product of two three cycles and each $5$-cycle and each product of two disjoint transpositions $(ab)(cd)$ is a product of two $3$-cycles then in fact every element in $\prod_{\omega} A_5$ is a product of exactly two conjugates of $g$ and we are done.
\end{example}

In all the above examples of Peano continua with comonster fundamental group, we used the fact that if the abelianization is uncountable then the fundamental group is comonster.

\begin{question}  Does there exist a Peano continuum whose first homology is trivial and whose fundamental group is comonster?
\end{question}

A negative answer would be very interesting as it would imply a theorem for finitely presented perfect groups (groups with trivial abelianization).

\begin{theorem}\label{finitelypresentedgroups}  Suppose the answer to the above question is no.  Let $\mathcal{P}_n$ be the class of groups whose elements are products of $n$ or fewer commutators.  For each $n\in \mathbb{N}$ there exists $k(n) \in \mathbb{N}$ such that if $G\in \mathcal{P}_n$ is finitely presented there exists a set $S \subseteq G$ with $\card(S) \leq k(n)$ and $\prod_{k(n)}S^{G} = G$ (each element of $G$ is a product of $k(n)$ or fewer conjugates of elements of $S$).

\end{theorem}

\begin{proof}  Suppose for contradiction that for some $n\in \mathbb{N}$ there is no such $k(n)$.  Select finitely presented groups $G_m \in \mathcal{P}_n$ such that for any $S\subseteq G_m$ with $\card(S) \leq m$ we have that $\prod_{m}S^{G_m} \neq G_m$.  For each $m$ there is a finite CW complex $Y_m$ of dimension at most two whose fundamental group is isomorphic to $G_m$.  Each such $Y_m$ is a Peano continuum.  Then $\prod_{m} Y_m$ is a Peano continuum, with fundamental group isomorphic to $\prod_m G_m$.  It is easy to see that $\prod_m G_m \in \mathcal{P}_n$ and is also comonster.
\end{proof}

This is adjacent to a question of Wiegold: Does every finitely generated perfect group contain an element which normally generates the group?

\end{section}

\begin{section}{Examples of Topologically Defined Subgroups}\label{Someexamples}

We give some standard examples, and introduce some new examples, of subgroups of the fundamental group which are topologically defined.  These are intended to illustrate the richness of the theory and give a grab bag of examples to which to apply the theorems.

\begin{subsection}{The Shape Kernel}

One well known subgroup of the fundamental group is the shape kernel.  We discuss this subgroup by first giving preliminary definitions towards defining the shape group and the shape kernel and then prove that the shape kernel is a closed subgroup.

We assume some familiarity with geometric simplicial complexes.  Given a topological space $X$ and a open cover $\mathcal{U}$
of $X$ let $N(\mathcal{U})$ denote the \textbf{nerve} of the cover- that is, the geometric simplicial complex which has a distinct vertex $v_U$ for every $U\in \mathcal{U}$ and which contains the $n$-simplex $[v_{U_0}, v_{U_1}, \ldots, v_{U_n}]$ if and only if $U_0\cap U_1\cap \cdots \cap U_n \neq \emptyset$.  If $\mathcal{V}$ is an open cover of $X$ that refines $\mathcal{U}$ (i.e. for each $V\in \mathcal{V}$ there is a $U\in \mathcal{U}$ such that $V\subseteq U$) then any map from the vertices of $N(\mathcal{V})$ to the vertices of $N(\mathcal{U})$ such that $v_V \mapsto v_U$ implies $V \subseteq U$ extends to a simplicial map from $N(\mathcal{V})$ to $N(\mathcal{U})$.

If the topological space has a distinguished basepoint $x$, then one can distinguish an element $U$ in an open cover $\mathcal{U}$ such that $x\in U$, which in turn gives a distinguished vertex in the nerve $N(\mathcal{U})$.  With this added structure, if $\mathcal{V}$ refines $\mathcal{U}$ with distinguished elements $V$ and $U$ such that $V \subseteq U$ then a simplicial map as  described above extending a vertex assignment satisfying $v_V \rightarrow v_U$, say $p_{(\mathcal{V}, V), (\mathcal{U}, U)}: (N(\mathcal{V}), v_V) \rightarrow (N(\mathcal{U}), v_U)$ preserves basepoint and is unique up to basepoint preserving homotopy.  Assuming $X$ is path connected all nerves are connected, and the refinement relation on open covers gives an inverse directed system $(\pi_1(N(\mathcal{U}), v_U), p_{(\mathcal{V}, V), (\mathcal{U}, U)}*)$.  The shape group of $X$ is defined as the inverse limit $$\check{\pi}_1(X,x) =\displaystyle \lim_{\leftarrow}(\pi_1(N(\mathcal{U}), v_U), p_{(\mathcal{V}, V), (\mathcal{U}, U)}*)$$

The index of the inverse limit will generally be of uncountable cardinality.  Assuming $(X, x)$ is also paracompact and Hausdorff we have for each open cover with distinguished neighborhood of $x$, $(\mathcal{U}, U)$, a refinement $(\mathcal{V}, V)$ such that $V$ is the unique element of $\mathcal{V}$ containing $x$.  A partition of unity subordinate to $\mathcal{V}$ (which necessarily exists by our assumption of paracompactness and Hausdorffness) induces a barycentric map $f_{\mathcal{U}}:(X, x) \rightarrow (N(\mathcal{V}), v_V)$ which is unique up to based homotopy.  The induced map $f_{\mathcal{U}*}$ on the fundamental group $\pi_1(X, x)$ can be checked to commute with the maps of the inverse system in the appropriate way, and since the set of all such open covers $\mathcal{V}$ described are cofinal in the inverse system we get a map $\Psi: \pi_1(X, x) \rightarrow \check{\pi}_1(X, x)$.

The natural object used to assess the loss of information when passing from the fundamental to the shape group is the shape kernel $\ker(\Psi)$.  The following demonstrates an alternative characterizaion of the shape kernel.

\begin{theorem}\label{shapekernel}  Suppose $(X, x)$ is a path connected metrizable space.  Then the shape kernel is a closed normal subgroup of $\pi_1(X, x)$.  If in addition $X$ is locally path connected then the shape kernel is equal to the following two subgroups:

\begin{enumerate}

\item $\displaystyle\bigcap_{f} \ker(f_*)$ where $f$ is taken over all continuous maps to semilocally simply connected spaces 

\item $\displaystyle \bigcap_{\text{G open, normal}} G$

\end{enumerate}

\end{theorem}

\begin{proof}  It is clear from the definition that the shape kernel is equal to $\bigcap_{f} \ker(f_*)$ where $f$ is taken over all baricentric maps of open covers.  Fix a barycentric map $f$ to the nerve $N(\mathcal{V})$.  As $N(\mathcal{V})$ is a geometric simplicial complex it is not difficult for each loop $l\in L_x$ to select $\epsilon >0$ such that for each $l'\in L_x$ that is $\epsilon$ close to $l$ we have $f\circ l$ is homotopic to $f\circ l'$ in $N(\mathcal{V})$.  This shows that $\ker(f_*)$ is open, and therefore also closed by Lemma \ref{openclosed}.  Then the shape kernel is a closed subgroup as an intersection of closed subgroups.

Now suppose that $X$ is also locally path connected.  Since each nerve is a geometric simplicial complex, each nerve is also semilocally simply connected.  Thus the shape kernel contains the subgroup (1).  Furthermore, if $f:X \rightarrow Y$ is continuous with $Y$ semilocally simply connected, then we can find an open cover $\mathcal{U}$ of $X$ such that the image of any loop in an element of $\mathcal{U}$ has nulhomotpic image under the map $f$.  This gives an open cover satisfying the criteria of Lemma \ref{opencover} and since $X$ is locally path connected we have that $\ker(f_*)$ is open.  Thus subgroup (1) contains subgroup (2).  

We conclude by proving that subgroup (2) contains the shape kernel.  Let $G$ be an open normal subgroup in $\pi_1(X)$ (since $G$ is open, normal we may consider $\pi_1(X, x)$ as basepoint free by Lemma \ref{pointindependence}).  Let $q: \pi_1(X) \rightarrow \pi_1(X)/G$ be the canonical quotient homomorphism.  We introduce some terms and a theorem given in \cite{CC}.

Recall that if $\phi: \pi_1(Y) \rightarrow H$ is a group homomorphism we say an open cover $\mathcal{V}$ of $Y$ by path connected sets is \textbf{2-set simple rel $\phi$} provided any loop whose image lies in the union of two elements of $\mathcal{V}$ is in $\ker(\phi)$ (as defined in the proof of Theorem \ref{Shelahgen}).  Two paths $p_0$ and $p_1$ are \textbf{$\mathcal{V}$-related} if there is some parametrization for $p_0$ and $p_1$ such that for all $s\in [0,1]$  the points $p_0(s)$ and $p_1(s)$ are in a common element of $\mathcal{V}$.  To be $\mathcal{V}$-related is not necessarily an equivalence relation; we say that paths $p_0$ and $p_1$ are \textbf{$\mathcal{V}$-equivalent} if they are in the same class under the equivalence class generated by $\mathcal{V}$-relatedness.  The following is part (1) of Theorem 7.3 in \cite{CC}:

\begin{theorem*}  Let $Y$ be a path connected topological space, $\phi:\pi_1(Y) \rightarrow H$ a homomorphism and $\mathcal{V}$ a 2-set simple cover of $Y$ rel $\phi$.  If two loops $l, l' \in L_y$ are $\mathcal{V}$-equivalent then $\phi([l]) = \phi([l'])$.

\end{theorem*}

By Lemma \ref{opencoverforopen} we have an open cover $\mathcal{U}_0$ of $X$ such that each loop in an element of $\mathcal{U}_0$ is in $G$.  For each $z\in X$ we may select $V_z\in \mathcal{U}_0$ satisfying $z\in V_z$.  Define $r_0(z) = d(z, X-V_z)$.  Letting $\mathcal{U}_1 = \{B(z, \frac{r_0(z)}{3})\}_{z\in X}$ it is straightforward to check that if for $U, U'\in \mathcal{U}_1$ we have $U \cap U' \neq \emptyset$ then $U \cup U'$ is contained in an element of $\mathcal{U}_0$.  By local path connectedness we let $\mathcal{U}_2$ be a refinement of $\mathcal{U}_1$ by path connected open sets.  It is easy to see that $\mathcal{U}_2$ is a 2-set simple cover rel $q$.  For each $z\in X$ pick a $W_z\in \mathcal{U}_2$ such that $z\in W_z$ and let $r_2(z) = d(z, X-W_z)$.  Letting $\mathcal{U}_3 = \{B(z, \frac{r_2(z)}{5})\}_{z\in X}$ one can check that if $U, U', U''\in \mathcal{U}_3$ satisfy $U\cap U' \neq \emptyset$ and $U'\cap U'' \neq \emptyset$ then $U\cup U'\cup U''$ is contained entirely in an element of $\mathcal{U}_2$.  Let $\mathcal{U}_4$ be a refinement of $\mathcal{U}_3$ by path connected open sets.  Finally for each $z\in X$ select a $U_z\in \mathcal{U}_4$ such that $z\in U_z$, let $r_4(z) = d(z, X-U_z)$ and $\mathcal{U} = \{B(z, \frac{r_4(z)}{3})\}_{z\in X}$.  Again, it is straightforward to see that if $U, U'\in \mathcal{U}$ satisfy $U\cap U' \neq \emptyset$ then $U\cup U'$ is entirely contained in an element of $\mathcal{U}_4$.  Without loss of generality we can assume $\mathcal{U}$ is refined so that $x$ is contained in exactly one element of the cover $\mathcal{U}$.

Let $b: X \rightarrow N(\mathcal{U})$ be a barycentric map associated to some partition of unity subordinated to $\mathcal{U}$.  Then $b(x) = v_U$ where $U\in \mathcal{U}$ is unique such that $x\in U$.  We define a map $f$ from the 1-skeleton $N(\mathcal{U})^1$ to $X$.  Let $f(v_U) = x$ and for all other vertices $v_{U'}\in N(\mathcal{U})^0$ simply let $f(v_{U'}) \in U'$.  By our choice of $\mathcal{U}$ if $[v_{U'}, v_{U''}]$ is a $1$-simplex in $N(\mathcal{U})$ then $U' \cap U'' \neq \emptyset$ and so there exists a path contained entirely in an element of $\mathcal{U}_4$ from $f(v_{U'})$ to $f(v_{U''})$.  Let $f|[v_{U'}, v_{U''}]$ map via this path.

We will be done if we show that $\ker(b_*) \leq G$.  Suppose now that $l\in L_x$ is such that $[l]\in \ker(b_*)$.  Then $b\circ l$ is a loop in $N(\mathcal{U})$ based at $v_U$ which is nulhomotopic.  Recall that $b$ has the property that $b^{-1}(\Star v_{U'}) \subseteq U'$ where $\Star v_{U'}$ is the open star of the vertex $v_{U'}$.  There exists a combinatorial loop $p(v_{U}, v_{U_1}, v_{U_2}, \ldots, v_{U_{n-1}}, v_{U_n} = v_{U})$ which is homotopic in $N(\mathcal{U})$ to $b\circ l$ such that $b\circ l(s) \in \Star v_{U_k}$ when $s\in [\frac{k}{n}, \frac{k+1}{n}]$.  Letting $l_0:[0,1] \rightarrow N(\mathcal{U})$ be a topological realization of this loop we see that $l$ is $\mathcal{U}_{2}$-related to $f\circ l_0$.

By assumption there exists a nulhomotopy of $l_0$, and so in particular there exists a combinatorial nulhomotopy of $p(v_{U}, v_{U_1}, v_{U_2}, \ldots, v_{U_{n-1}}, v_{U_n} = v_{U})$.  In other words, there exists a finite sequence of combinatorial paths:

\begin{center}  $p_0 = p(v_{U}, v_{U_1}, v_{U_2}, \ldots, v_{U_{n-1}}, v_{U_n} = v_{U})$

$p_1 = p(v_U, v_{U_{1,1}},v_{U_{1, 2}}, \ldots, v_{U_{1, n_1}} = v_U)$

$p_2 = p(v_U, v_{U_{2, 1}}, v_{U_{2, 2}}, \ldots, v_{U_{2, n_2}} = v_U)$

$\vdots$

$p_{m}=p(v_U)$

\end{center}

such that one obtains $p_k$ from $p_{k-1}$ by performing one of the following elementary path homotopies:

\begin{enumerate}  \item Exchanging the subpath $v_{U_p}, v_{U_{p+1}}$ for the subpath $v_{U_p}$ assuming $U_p = U_{p+1}$, or vice versa.

\item Exchanging the subpath $v_{U_p}, v_{U_{p+1}}, v_{U_{p+2}}$ for the subpath $v_{U_p}$ assuming $U_{p+2} = U_p$, or vice versa.

\item Exchanging the subpath $v_{U_p}, v_{U_{p+1}}, v_{U_{p+2}}$ for the subpath $v_{U_p}, v_{U_{p+2}}$ assuming $[v_{U_p}, v_{U_{p+1}}, v_{U_{p+2}}]$ is a $2$-simplex in $N(\mathcal{U})$, or vice versa.
\end{enumerate}

Letting $l_k:[0,1] \rightarrow N(\mathcal{U})$ be a topological realization of the combinatorial path $p_k$, it is easy to see that $f\circ l_k$ is $\mathcal{U}_{2}$ related to $f\circ l_{k+1}$.  By the theorem of Cannon and Conner quoted above, we have that $q([l]) = q([f\circ l_0]) = q([f\circ l_1]) = \cdots = q([f\circ l_m]) = q(1)$, and so $[l]\in G$.
\end{proof}

As a direct consequence of Theorem \ref{CH} we get the following: if $X$ is a Polish space then the quotient of $\pi_1(X)$ by the shape kernel is of cardinality $\leq \aleph_0$ or $2^{\aleph_0}$.

\end{subsection}

\begin{subsection}{The Spanier Group}\label{Spanier}  Another useful subgroup of the fundamental group is the Spanier group, which we denote $\pi_1^s(X, x)$ (first defined in \cite{Sp}).  We give the necessary definitions for this group, then give some results about the topological properties.

Let $X$ be a path connected topological space and $x\in X$.  If $\mathcal{U}$ is a collection of open subsets of $X$ we define $\pi_1(\mathcal{U}, x)$ to be the subgroup of $\pi_1(X, x)$ generated by loops of the form $\rho * l * \rho^{-1}$ where $\rho(0) = x$ and $l$ is a loop based at $\rho(1)$ and contained in some element of $\mathcal{U}$.  This subgroup is easily seen to be normal.  The Spanier group is defined to be $\pi_1^s(X, x) = \bigcap_{\mathcal{U}}\pi_1(\mathcal{U}, x)$ where the parameter $\mathcal{U}$ is taken over all open covers.  The first of the following two lemmas does not assume metrizability of $X$.  It is proven in \cite{FZ} as Proposition 4.8.  We provide our own proof for completeness.

\begin{lemma}\label{Spanierlemma0}  $\pi_1^s(X, x)$ is contained in the shape kernel.
\end{lemma}

\begin{proof}  Let $b$ be a barycentric map from $X$ to some nerve.  Since a nerve is semilocally simply connected we have an open cover $\mathcal{U}$ of $X$ such that any loop contained in an element of $\mathcal{U}$ is in $\ker(b_*)$.  Obviously $\pi_1(\mathcal{U}, x) \leq \ker(b_*)$ and taking the appropriate intersections gives the claim.
\end{proof}

\begin{lemma}\label{Spanierlemma}  Let $X$ be a metric space, $\mathcal{U}$ an open cover of $X$ and $x\in X$.

\begin{enumerate}  \item If $X$ is locally path connected then $\pi_1(\mathcal{U}, x)$ is open.

\item If $X$ is Polish and $\mathcal{U}$ is countable then $\pi_1(\mathcal{U}, x)$ is analytic.

\end{enumerate}

\end{lemma}

\begin{proof}  Assume the hypotheses for part (1).  The open cover $\mathcal{U}$ is such that any loop contained in an element thereof (considering loops to be base point free) is an element of $\pi_1(\mathcal{U})$ (we switch here to a basepoint free notation for emphasis).  Then by Lemma \ref{opencover} we have that $\pi_1(\mathcal{U})$ is an open subgroup.

Assume the hypotheses for part (2).  Let $L_{x, U, n} = \{l\in L_x: (\forall s\in [0,\frac{1}{2}])[l(s) = l(1-s)] \wedge (\forall s\in [\frac{1}{3}, \frac{2}{3}])[d(l(s), X-U) \geq \frac{1}{n}]\}$ where $U\in \mathcal{U}$ and $n\in \omega$.  It is clear that $L_{x, U, n}$ is closed as a subset of $L_x$.  The set $\bigcup_{U\in \mathcal{U}, n\in \omega}L_{x, U, n}$ is a countable union of closed sets (and therefore analytic).  Then $\pi_1(\mathcal{U}, x) = \langle \langle \bigcup_{U\in \mathcal{U}, n\in \omega}L_{x, U, n} \rangle \rangle$ is analytic by Lemma \ref{normalclosure}.
\end{proof}

That the shape kernel is equal to the Spanier group  for all locally path connected, path connected paracompact Hausdorff spaces was recently shown in \cite{BF}.  Part (1) of the following theorem gives a rather short proof of a slightly less general fact.

\begin{theorem}\label{Spaniertheorem}  The following hold:

\begin{enumerate}
\item If $X$ is a locally path connected metric space then $\pi_1^s(X, x)$ is equal to the shape kernel, and in particular closed.

\item If $X$ is a compact metric space then $\pi_1^s(X, x)$ is analytic.
\end{enumerate}

\end{theorem}
\begin{proof}  (1)  Assume the hypotheses.  That the Spanier group is contained in the shape kernel was proved in Lemma \ref{Spanierlemma0}.  That the shape kernel is contained in the Spanier group follows from characterization (2) of Theorem \ref{shapekernel} and from Lemma \ref{Spanierlemma} part (1).

For (2) we assume the hypotheses.  As $X$ is a compact metric space there exists a sequence $\{\mathcal{U}_n\}_{n\in \omega}$ of finite open covers such that $\mathcal{U}_{n+1}$ refines $\mathcal{U}_n$ and which is cofinal in the inverse directed system of open covers.  Thus $\pi_1^s(X, x) = \bigcap_{n\in \omega} \pi_1(\mathcal{U}_n, x)$ is analytic as a countable intersection of analytic subgroups (Lemma \ref{Spanierlemma} part (2) and Theorem \ref{closureprop} part (4)).
\end{proof}
\end{subsection}

\begin{subsection}{Subgroups reflecting local behavior}  We give a couple of subgroups that can be thought of as indicating local behavior.  First, recall that a space $X$ is \textbf{homotopically Hausdorff at $x$} if each loop based at $x$ which can be homotoped into any neighborhood of $x$ is nulhomotopic.  This notion has found many uses (for example in \cite{BS} and \cite{FZ}).  If $X$ is a Polish space, let $L_{x, n}$ be the set of all loops given by $l \in L_{x, n}$ if and only if $(\forall s\in [0,1])[d(l(s), x) \leq \frac{1}{n}]$.  Then $L_{x, n}$ is clearly a closed subset of $L_x$, so the subgroup $\langle [L_{x, n}]\rangle$ is analytic by Lemma \ref{loopgenerationsubgroup}.  The subgroup $\bigcap_{n\in \omega} \langle [L_{x, n}]\rangle$ is trivial if and only if $X$ is homotopically Hausdorff at $x$.  This subgroup is analytic and can be thought of as the indicator subgroup for the property.

We give another example of a subgroup reflacting local behavior.  If $X$ is compact, metrizable and path connected, then it is easy to see that the cone over $X$, $\mathcal{C}X =  X \times[0,1]/X\times \{1\}$, is also compact, metrizable and path connected.  We shall consider $X$ as a subset of $\mathcal{C}X$ by identifying $X$ with $X\times \{0\}$.

Let $S\subseteq X$ be nonempty.  Fixing a metric on $\mathcal{C}X$ we let $Y_{n, S} \subseteq \mathcal{C}X$ be given by $Y_{n,S} = X \cup (\mathcal{C}X \setminus B(S, \frac{1}{n}))$.  Let $f_{n, S}$ be the inclusion map from $X$ to $Y_{n, S}$.  Then $f_{n, S}$ is a continuous map to a compact metric space, and $\ker(f_{n, S*})$ is analytic.  Since $Y_{n, S} = Y_{n, \overline{S}}$ there is no generality lost in assuming that $S$ is compact.  Also, the choice of metric on $\mathcal{C}X$ does not change $\bigcup_{n} \ker(f_{n, S*})$ (by compactness).  Let $N(S)$ denote the normal subgroup $\bigcup_{n} \ker(f_{n, S*})$.  This subgroup is intended to convey a sense of the importance of the subspace $S$ in the fundamental group of $X$.  If the subgroup $N(S)$ is all of $\pi_1(X)$ then the points of $S$ carry little significance in the fundamental group.  If $N(S)$ is trivial, then the points of $S$ can be thought of as holding importance.  If $S \subseteq S'$ then $Y_{n, S} \supseteq Y_{n, S'}$ and so $N(S') \leq N(S)$.

\begin{example}  Let $X$ be compact, metrizable and path connected.  Letting $S = X$ we get that for every $n\in \omega \setminus \{0\}$, the path component in $Y_{n, S}$ including all elements of $X$ is simply the subset $X$.  Thus any nulhomotopy of a loop in $X$ taking place in $Y_{n, S}$ must in fact already take place in $X$, so $N(S)$ is trivial.
\end{example}

\begin{example}  Let $S\subseteq X$ be a compactum such that any map $f:S^1 \rightarrow X$ can be homotoped to have image disjoint from $S$.  Then given $x\in X$ and a loop $l\in L_x$ there is a homotopy of $l$ to a loop $\rho * l'*\rho^{-1}$ such that $l'$ is a loop with image disjoint from $S$.  By compactness there is some positive distance between $S$ and the image of $l'$, and so $l'$ can be nulhomotoped in $Y_{n,S}$ for some $n$, so that $l$ is also nulhomotopic in $Y_{n,S}$.  Then $N(S) = \pi_1(X)$.
\end{example}

\begin{example}  Let $X= S^1$ and $S = \{x\}$ be any singleton.  For each $n\in \omega \setminus \{0\}$ there is a superset $Z \supseteq Y_{n, S}$ such that $Z$ strongly deformation retracts to the set $X$, so that $N(S)$ is trivial.  This holds true as well if $X$ is a wedge of finitely many circles and $x$ is the wedge point by the same proof.
\end{example}

\begin{lemma}\label{retractionlemma}  If $r: X \rightarrow Y$ is a retraction with $Y\supset S$ then the monomorphism induced by inclusion $\pi_1(Y) \rightarrow \pi_1(X)$ induces a monomorphism $\pi_1(Y)/N_Y(S) \rightarrow \pi_1(X)/N_X(S)$ (here we use the subscript to denote the ambient space).
\end{lemma}

\begin{proof}  This follows from the fact that the retraction $r$ extends to a retraction $R$ of the cones $R: C(X)\rightarrow C(Y)$ given by $R(x, t) = (r(x), t)$ where $t\in [0,1]$.
\end{proof}

\begin{example}  Let $E$ be the Hawaiian earring (see Example \ref{firsthawaiianearring}) and $S=\{x\}$ where $x$ is the wedge point.  The wedge $Y_m$ of the outer $m$ circles is a retract of $X$ and each $N_{Y_m}(S)$ is trivial by the previous example.  Then $N_E(S)$ has no elements of the canonical free group retracts.  Then $N_E(S)$ is trivial by the standard fact that the Hawaiian earring fundamental group injects naturally into the inverse limit of the canonical free subgroups.
\end{example}

\begin{example}  Consider the Hawaiian earring $E$ again and $S = \{x\}$ with $x$ any other point in $E$ besides the wedge point.  Then for some $n\in \omega\setminus \{0\}$ the ball $B(x, \frac{1}{n})$ does not intersect any other circle on the Hawaiian earring besides that on which $x$ lies.  Then $N(S)$ contains the kernel of the retraction induced homomorphism $r_*$ where $r$ fixes the circle on which $x$ lies and takes all other points to the wedge point.  On the other hand, $N(S)$ must be precisely the kernel of the induced homomorphism by Lemma \ref{retractionlemma}.
\end{example}

\end{subsection}

\end{section}

\begin{section}{N-slenderness, products and free products}\label{freeprod}

In this section we introduce n-slender groups (see \cite{E1}).  This will require an understanding of the fundamental group of the Hawaiian earring $E$.  Recall that the Hawaiian earring is the compact subspace $E=\bigcup_{n\in \omega} C((0, \frac{1}{n+2}),\frac{1}{n+2})$ of $\mathbb{R}^2$, where $C(p, r)$ is the circle centered at $p$ of radius $r$.  The space $E$ can be thought of as a shrinking wedge of countably infinitely many circles.  The fundamental group $\pi_1(E)$ has a combinatorial characterization which we describe below.

We let $\{a_n^{\pm 1}\}_{n=0}^{\infty}$ be a countably infinite set with formal inverses.  A map $W: \overline{W}\rightarrow \{a_n^{\pm 1}\}_{n=0}^{\infty}$ from a countable totally ordered set $\overline{W}$ is a \textbf{word} if for every $n\in \omega$ the set $W^{-1}(\{a_n^{\pm 1}\})$ is finite.  We say two words $U$ and $V$ are isomorphic, $U \simeq V$, provided there is an order isomorphism of the domains of each word $f: \overline{U} \rightarrow \overline{V}$ such that $U(t) = V(f(t))$.  We identify isomorphic words.  The class of isomorphism classes of words is a set of cardinality continuum which we denote $\W$.  For each $N\in \omega$ define the projection $p_N$ to the set of finite words by letting $p_N(W) = W|\{t\in \overline{W}: W(t) \in \{a_n^{\pm 1}\}_{n=0}^{N}\}$.  Define an equivalence relation $\sim$ on words as follows: given words $U, V\in \W$ we let $U \sim V$ if for each $N\in \omega$ we have $p_N(U) = p_N(V)$ in the free group $F(\{a_0, \ldots, a_N\})$.  For each word $U$ there is an inverse word $U^{-1}$ whose domain is the totally ordered set $\overline{U}$ under the reverse order and $U^{-1}(t) = U(t)^{-1}$.  Given two words $U, V\in \W$ we form the concatenation $UV$ by taking the domain of $UV$ to be the disjoint union of $\overline{U}$ with $\overline{V}$, with order extending that of $\overline{U}$ and $\overline{V}$ and placing all elements of $\overline{U}$ before those of $\overline{V}$, and $UV(t) = \begin{cases}U(t)$ if $t\in \overline{U}\\V(t)$ if $t\in \overline{V}  \end{cases}$.  The set $\W/\sim$ is endowed with a group structure with binary operation given by $[U][V] = [UV]$, inverses defined by $[U]^{-1} = [U^{-1}]$ and the equivalence class of the empty word being the trivial element.

Letting $\HEG$ denote the group $\W/\sim$, the free group $F(\{a_0, \ldots, a_N\})$, which we shall denote $\HEG_N$, may be though of as a subgroup in $\HEG$.  The word map $p_N$ gives a group retraction $\HEG \rightarrow \HEG_N$ which we also denote $p_N$.  The word map $p^N$ given by the restriction $p^N(W) = W|\{t\in \overline{W}: W(t) \in \{a_n^{\pm 1}\}_{n=N+1}^{\infty}\}$ induces another group retraction from $\HEG$ to the subgroup $\HEG^N$ consisting of those equivalence classes which contain words involving no letters in $\{a_n^{\pm 1}\}_{n=0}^N$.  Let $p^N$ denote this group retraction.  By considering a word $W$ as a concatenation of finitely many words in the letters $\{a_n^{\pm 1}\}_{n=0}^N$ and finitely many words in the letters $\{a_n^{\pm 1}\}_{n=N+1}^{\infty}$ we obtain an isomorphism $\HEG \simeq \HEG_N \ast \HEG^N$.  The homomorphism $p_N$ corresponds to the topological retraction of $E$ to the subspace $\bigcup_{n\leq N} C((0, \frac{1}{n+2}),\frac{1}{n+2})$ and similarly for $p^N$ and $\bigcup_{n>N} C((0, \frac{1}{n+2}),\frac{1}{n+2})$.  We are now ready for the following definition:

\begin{definition}  A group $G$ is noncommutatively slender (or n-slender) if for each homomorphism $\phi: \HEG \rightarrow G$ there exists $N\in \omega$ such that $\phi = \phi\circ p_N$.
\end{definition}

This definition was first introduced by K. Eda in \cite{E1}.  The additive group on $\mathbb{Z}$ was the first nontrivial group known to be n-slender \cite{H}, and Eda has shown that the class of n-slender groups is closed under free-products and direct sums (see \cite{E1}).  Torsion-free word hyperbolic groups are known to be n-slender \cite{Co}.  For each infinite cardinal $\kappa$ there exists an n-slender group of cardinality $\kappa$ (for example, the free group of rank $\kappa$).  We give the following alternative characterization of n-slender groups before moving on to the theorems associated with this section:

\begin{lemma}\label{altnslender}  A group $G$ is n-slender if and only if for every  locally path connected metric space $X$ each homomorphism $\phi:\pi_1(X) \rightarrow G$ has open kernel.
\end{lemma}

\begin{proof}  For the $\Rightarrow$ direction we suppose that $G$ is n-slender and that $\phi: \pi_1(X) \rightarrow G$ is a homomorphism, with $X$ a metric path connected, locally path connected space.  Letting $x\in X$ we claim that for some $\epsilon>0$ all loops (not necessarily based at $x$) in the open ball $B(x, \epsilon)$ are in the kernel of $\phi$.  Were this not the case there would exist a sequence of loops $\{l_n\}_{n\in \omega}$ such that $\diam(\{x\}\cup l_n([0,1])) \leq 2^{-n}$ and $l_n$ is not in the kernel of $\phi$.  By local path connectedness we may pass to a subsequence and eventually attach the bases of the $l_n$ to $x$ via a small path.  Thus we may assume without loss of generality that the $l_n$ are all based at $x$.  Define a map $f:E \rightarrow X$ by mapping the circle $C((0, \frac{1}{n+2}),\frac{1}{n+2})$ along the loop $l_n$ so that a generator of $\pi_1(C((0, \frac{1}{n+2}),\frac{1}{n+2}))$ maps to $[l_n] \in \pi_1(X, x)$ under the restriction $f_*|\pi_1(C((0, \frac{1}{n+2}),\frac{1}{n+2}))$.  Now $\phi\circ f_*$ is a map from $\HEG$ to $G$ and so for some $N$ we have $n\geq N$ implies $\phi\circ f_*(a_n) = 1$.  But $a_n$ corresponds to one of the two generators of $\pi_1(C((0, \frac{1}{n+2}),\frac{1}{n+2}))$, so that $1 = \phi\circ f_*(a_n) = \phi([l_n])$, a contradiction.  Thus such an $\epsilon$ must exist, and we get an open cover satisfying the hypotheses of Lemma \ref{opencover} for the subgroup $\ker(\phi)$, so that $\ker(\phi)$ is open.

For the direction $\Leftarrow$ we let $G$ be a group such that for every locally path connected metric space $X$ every homomorphism $\phi:\pi_1(X) \rightarrow G$ has open kernel.  Letting $E = X$ and $\phi:\HEG \rightarrow G$ be a homomorphism, $\ker(\phi)$ is an open subgroup of $\pi_1(E)$.  By Lemma \ref{opencoverforopen} there exists some $\epsilon>0$ such that any loop in $B((0,0), \epsilon) \subseteq E$ is in $\ker(\phi)$.  Selecting $N\in \omega$ such that $\epsilon>\frac{2}{N+2}$, we have $\phi|\HEG^N$ is trivial, so that $\phi = \phi \circ p_N$.  Thus $G$ is n-slender.
\end{proof}

Recall that a space is $\kappa$-Lindel\"of if every open cover of the space contains a subcover of cardinality at most $\kappa$.  It is easily seen that a metric space $Z$ is $\kappa$-Lindel\"of if and only if $Z$ has a dense subset of cardinality $\leq \kappa$ if and only if $Z$ has a basis of cardinality $\leq \kappa$.  It is also true that if $X$ is a metric $\kappa$-Lindel\"of space then so is $L_x$ for each $x\in X$.  Lemma \ref{altnslender} has the following easy corollary:

\begin{corollary}\label{imageinnslender}  If $X$ is locally path connected metrizable $\kappa$-Lindel\"of and $G$ is n-slender then the image of any homomorphism $\phi:\pi_1(X) \rightarrow G$ has $\card(\phi(G)) \leq \kappa$.
\end{corollary}

\begin{proof}  Assume the hypotheses.  Then $\ker(\phi)$ is open, and by Lemma \ref{relations} the equivalence relation given by left cosets of $\ker(\phi)$ has open equivalence classes.  As $L_x$ has a dense subset of cardinality $\leq \kappa$, we see that $\card(\pi_1(X)/\ker(\phi))\leq \kappa$.
\end{proof}

By considering a wedge of $\kappa$ circles, each circle of diameter $1$, endowed with the path metric, one has an example of a $\kappa$-Lindel\"of space which is completely metrizable, locally path connected whose fundamental group is free of rank $\kappa$.  Thus the conclusion of Corollary \ref{imageinnslender} cannot be strengthened.

One can prove other results which give obstructions to surjections from the fundamental group, even when the codomain is not n-slender, such as the following: 

\begin{theorem}\label{productofslender}  Suppose $X$ is a locally path connected $\kappa$-Lindel\"of metric space, $\{G_i\}_{i\in I}$ is a collection of n-slender groups, and $\phi:\pi_1(X) \rightarrow \prod_{i\in I}G_i$ is a homomorphism.  Then there exists some $I' \subseteq I$ with $\card(I') \leq \kappa$ such that $\ker(p_{I'}\circ \phi) = \ker(\phi)$. (Here the map $p_{I'}:\prod_{i\in I}G_i \rightarrow \prod_{i\in I'}G_i$ is projection.)
\end{theorem}

This immediately yields:

\begin{theorem}\label{nosurjection}  If $X$ is a locally path connected $\kappa$-Lindel\"of metric space and $\{G_i\}_{i\in I}$ is a collection of nontrivial n-slender groups with $\card(I) >\kappa$ then there is no epimorphism $\phi:\pi_1(X) \rightarrow \prod_{i\in I} G_i$.
\end{theorem}

\begin{example}  Theorem \ref{nosurjection} need not hold if local path connectedness is dropped.  For example there exists a model $\mathcal{N}$ of set theory satisfying

\begin{enumerate}\item \textbf{ZFC}

\item $2^{\aleph_0} = \aleph_2$

\item $(\forall \kappa \geq \aleph_1)[2^{\kappa}=\kappa^+]$

\end{enumerate}

\noindent (see \cite{Be}, 2.19).  We consider the space $F$ from Example \ref{exampleF}. Since $\pi_1(X)$ is a free group of rank $2^{\aleph_0}$ and the group $\prod_{\aleph_1}\mathbb{Z}$ is of cardinality $2^{\aleph_1} = 2^{\aleph_0}$, there exists a surjection from $\pi_1(X)$ to $\prod_{\aleph_1}\mathbb{Z}$.

Theorem \ref{nosurjection} also fails in the model $\mathcal{N}$ if the hypothesis that the $G_i$ are n-slender is dropped.  We consider the space $P^{\omega}$ from Example \ref{productofprojectiveplanes} in the model $\mathcal{N}$.  We have $\pi_1(P^{\omega}, \overline{x}) \simeq \prod_{\omega}\mathbb{Z}/2\mathbb{Z} \simeq \bigoplus_{2^{\aleph_0}}\mathbb{Z}/2\mathbb{Z} \simeq \prod_{\aleph_1}\mathbb{Z}/2\mathbb{Z}$.

That Theorem \ref{nosurjection} holds is a nontrivial fact, since $\pi_1(X)$ and $\prod_{i\in I}G_i$ can have the same cardinality (as would happen in the model $\mathcal{N}$ above).  In a model of \textbf{ZFC} where the generalized continuum hypothesis holds, Theorem \ref{nosurjection} would hold without the local path connectedness assumption or the n-slenderness of the $G_i$ simply by noticing that

\begin{center}
$\card(\pi_1(X)) \leq \card(L_x) \leq\card(X)^{\aleph_0}\leq (\kappa^{\aleph_0})^{\aleph_0}$

$\leq \kappa^+ = 2^{\kappa}$

$\leq \card(I) < 2^{\card(I)} \leq \card(\prod_{i\in I}G_i)$

\end{center}
\end{example}

\begin{proof}(of Theorem \ref{productofslender})  Assume the hypotheses hold and the conclusion fails.  Let $p_j:\prod_{i\in I}G_i \rightarrow G_j$ denote projection to the $j$-th coordinate.  Let $x\in X$ and $\mathcal{B}$ be a basis for the topology on $L_x$ with $\card(\mathcal{B}) \leq \kappa$ and $\emptyset \notin \mathcal{B}$.  Pick $i_0\in I$ such that $\ker(p_{i_0}\circ \phi) \neq \ker(\phi)$.  Suppose we have defined $i_\alpha$ for all $\alpha<\beta < \kappa^+$ so that for all $\gamma_0 < \gamma_1<\beta$ we have that $\ker(p_{\{i_{\alpha}\}_{\alpha\leq \gamma_0}} \circ \phi)$ is a proper superset of $\ker(p_{\{i_{\alpha}\}_{\alpha\leq \gamma_1}} \circ \phi)$.  Select $i_{\beta}$ so that $\ker(p_{i_{\beta}}\circ \phi)$ does not contain $\ker(p_{\{i_{\alpha}\}_{\alpha<\beta}} \circ \phi)$.  Such a selection is possible since we assume that  $\ker(p_{I'}\circ \phi) \neq  \ker(\phi)$ for all $I'$ such that $\card(I')\leq \kappa$.

Now each $\ker(p_j\circ \phi)$ is an open subgroup of $\pi_1(X, x)$ by Lemma \ref{altnslender}, and so is closed by Lemma \ref{openclosed}.  As it is clear that $\ker(p_{I'}\circ \phi) = \bigcap_{j\in I'}\ker(p_j\circ \phi)$ for any $I'\subseteq I$ we know that any $\ker(p_{I'}\circ \phi)$ is closed.  Pick $O_0\in \mathcal{B}$ such that $O_0 \cap \ker(p_{i_0}\circ \phi) \neq \emptyset$ and $O_0\cap \ker(p_{\{i_0, i_1\}}\circ \phi) = \emptyset$.  For $0<\beta<\kappa^+$ select $O_{\beta}\in \mathcal{B}$ such that $O_{\beta}\cap \ker(p_{\{i_{\alpha}\}_{\alpha\leq\beta}} \circ \phi) \neq \emptyset$ and $O_{\beta}\cap \ker(p_{\{i_{\alpha}\}_{\alpha\leq \beta+1}} \circ \phi) = \emptyset$.  The $O_{\beta}$ are pairwise distinct for different indices, so the map $\kappa^+ \rightarrow \mathcal{B}$ given by $\beta \mapsto O_{\beta}$ is an injection, a contradiction.
\end{proof}

We move on to a result on free products of groups.  We first state the following instance of Theorem 1.3 in \cite{E2}:

\begin{corollary*}  Suppose $\phi:\HEG \rightarrow *_{i\in I}G_i$ is a homomorphism from $\HEG$ to a free product.  Then there exists $N\in \omega$, $g\in*_{i\in I}G_i$  and $j\in I$ such that $\phi(\HEG^N) \leq gG_jg^{-1}$.
\end{corollary*}

We use this to prove the following:

\begin{lemma}\label{coverforfreeprod}  Suppose $X$ is a first countable, locally path connected space and $\phi:\pi_1(X) \rightarrow  *_{i\in I}G_i$ is a homomorphism.  For each $x\in X$ there exists a path connected neighborhood $B_x$ and $j\in I$ such that $\phi(\pi_1(B_x, x))$ is contained in a conjugate of $G_j$.
\end{lemma}

\begin{proof}  We first notice that the statement of the lemma actually makes sense, because any change of basepoint would only alter $\phi$ by conjugation.  Thus we may assume that the domain of $\phi$ is in fact $\pi_1(X, x)$.  By $\phi(\pi_1(B_x, x))$ we understand the image of the composition of the map induced by inclusion $\iota: B_x \rightarrow X$ with $\phi$.

Assuming the lemma is false there exists a sequence of path connected neighborhoods $\{U_n\}_{n\in \omega}$ with $\bigcap_{n\in \omega} U_n = \{x\}$ and loops $\{l_n\}_{n\in \omega}$ based at $x$ such that the image of $l_n$ is contained in $U_n$ and for every $n$ if $\phi([l_n]) \in gG_jg^{-1}$ there exists $n_0>n$ with $\phi([l_{n_0}]) \notin gG_jg^{-1}$.  As in the proof of Lemma
\ref{altnslender} we define a map $g:E \rightarrow X$ so that the $n$-th circle in $E$ traces out $l_n$.  Then $\phi\circ g_*$ violates the corollary.  
\end{proof}

Lemma \ref{coverforfreeprod} yields the following theorem, which is similar in flavor to Theorem \ref{productofslender} but with no mention of n-slender groups:

\begin{theorem}\label{freeprodsplitting}  Suppose $\phi:\pi_1(X) \rightarrow *_{i\in I}G_i$ is a homomorphism, with $X$ a locally path connected $\kappa$-Lindel\"of metric space.  Then for some $I' \subseteq I$ with $\card(I')\leq \kappa$ we have $\phi(\pi_1(X)) \leq *_{i\in I'}G_i$.
\end{theorem}

\begin{proof}  By the Kurosh subgroup theorem \cite{Ku} we have $$\phi(\pi_1(X)) = F(J) * (*_{m \in M} g_m H_{j_m} g_m^{-1})$$ where $j_m\in I$, $H_{j_m} \leq G_{j_m}$ and $F(J)$ is a free group generated by $J\subseteq  *_{i\in I}G_i$.  Letting $r:F(J) * (*_{m \in M} g_m H_{j_m} g_m^{-1}) \rightarrow F(J)$ be the obvious retraction we notice that as $r\circ \phi$ is a map to an n-slender group, $\ker(r\circ \phi)$ is open in $\pi_1(X)$ by Lemma \ref{altnslender}.  By Lemma \ref{opencoverforopen} we have an open cover $\mathcal{U}_0$ of $X$ such that any loop in an element of $\mathcal{U}_0$ is in $\ker(r\circ \phi)$.  Let $\mathcal{U}_1$ be a refinement of $\mathcal{U}_0$ such that $U, V\in \mathcal{U}_1$ with $U\cap V \neq \emptyset$ implies $U \cup V$ is contained in an element of $\mathcal{U}_0$.  Let $\mathcal{U}$ be a refinement of $\mathcal{U}_1$ consisting of path connected open sets.  Since $X$ is $\kappa$-Lindel\"of we may assume $\card(\mathcal{U}) \leq \kappa$.  The following is a statement of Theorem 7.3 of \cite{CC} part (2) (see proof of Theorem \ref{Shelahgen} for definition of $2$-set simple):

\begin{theorem*}  If $\psi:\pi_1(X) \rightarrow H$ is a group homomorphism and $\mathcal{U}$ is a $2$-set simple cover rel $\psi$ with nerve $N(\mathcal{U})$ then $\psi(\pi_1(X))$ is a factor group of $\pi_1(N(\mathcal{U}))$.
\end{theorem*}

It is clear that $\mathcal{U}$ is $2$-set simple rel $r\circ \phi$, so $F(J)$ is the homomorphic image of $\pi_1(N)$, and since $N$ has only $\kappa$-many vertices we know $\kappa \geq \card(\pi_1(N)) \geq \card(F(J))$.

We show that $M$ is of cardinality at most $\kappa$.  Then we will let $I_0 = \{j_m\}_{m\in M} \cup \bigcup_{m\in M} I_m$ where $g_m \in *_{j\in I_m}G_j$ and $I_1$ be such that $F(J) \leq *_{j\in I_1}G_j$, where each $I_m$ is finite and $\card(I_1) \leq \kappa$.  Thus $I' = I_0 \cup I_1$ will be of cardinality $\leq \kappa$ and clearly $\phi(\pi_1(X)) \leq *_{i\in I'}G_i$..

We show $\card(M) \leq \kappa$ by demonstrating that if $\psi: \pi_1(X) \rightarrow *_{m\in M} \Gamma_m$ is onto, with each group $\Gamma_m$ nontrivial, then $\card(M) \leq \kappa$.  By Lemma \ref{coverforfreeprod} we can obtain a cover $\mathcal{V}_0$ of $X$  by open balls such that each loop with image in an element of $\mathcal{V}_0$ maps to a conjugate of one of the $\Gamma_m$.     As $X$ is $\kappa$-Lindel\"of we may assume $\card(\mathcal{V}_0) \leq \kappa$.  For each $B\in \mathcal{V}_0$ select an $m_B\in M$ such that $\pi_1(B)$ maps under $\phi$ to a conjugate of $\Gamma_{m_B}$ and let $M' = \{m_B\}_{B\in \mathcal{V}_0}$.  Then $\card(M') \leq \kappa$.  We show that $\card(M \setminus M') \leq \kappa$ and we shall be done.  Let now $r:*_{m\in M} \Gamma_m \rightarrow *_{m\in M\setminus M'}\Gamma_m$ be the obvious retraction.  Refine $\mathcal{V}_0$ to an open cover $\mathcal{V}$ by path connected open sets such that $U \cap V \neq \emptyset$ implies $U \cup V$ is included in some element of $\mathcal{V}_0$ and $\card(\mathcal{V}) \leq \kappa$.  Now the cover $\mathcal{V}$ is $2$-set simple rel the map $r\circ \psi:\pi_1(X) \rightarrow  *_{m\in M\setminus M'}\Gamma_m$, so the image of $r\circ \psi$ is a homomorphic image of the nerve of $\mathcal{V}$, and so $r\circ \psi$ has image of cardinality $\leq \kappa$.  Then $\card(M\setminus M') \leq \kappa$ and we are done. 
\end{proof}

The following corollary is immediate:
\begin{corollary}\label{nouncountabledecomp}  If $X$ is a locally path connected separable metric space then $\pi_1(X)$ is not a free product of uncountably many nontrivial groups.  More generally if $X$ is a locally path connected $\kappa$-Lindel\"of metric space then $\pi_1(X)$ is not a free product of $>\kappa$ many nontrivial groups.
\end{corollary}

The comparable result for a compact space holds as well:

\begin{theorem}\label{freeprodforcompacta}  If $X$ is a Peano continuum and $\phi:\pi_1(X) \rightarrow *_{i\in I}G_i$ is a homomorphism, then for some finite $I'\subseteq I$ we have $\phi(\pi_1(X)) \leq *_{i\in I'}G_i$.
\end{theorem}

\begin{proof}  The proof runs the same as the proof of Theorem \ref{freeprodsplitting} except that the images of the fundamental groups of the nerves of the covers become finitely generated.  Thus $F(J)$ is finite rank and the $M$ in the corresponding claim is finite for the same reason.
\end{proof}

Lemma \ref{coverforfreeprod} also yields the following result for Polish spaces:

\begin{C}  Suppose $X$ is locally path connected Polish and $\pi_1(X) \simeq *_{i\in I}G_i$ with each $G_i$ nontrivial.  The following hold:

\begin{enumerate}\item$\card(I)\leq \aleph_0$

\item Each retraction map $r_j:*_{i\in I}G_i \rightarrow G_j$ has analytic kernel.

\item  Each $G_j$ is of cardinality $\leq \aleph_0$ or $2^{\aleph_0}$.

\item  The map $ *_{i\in I}G_i \rightarrow \bigoplus_{i\in I}G_i$ has analytic kernel.

\end{enumerate}
\end{C}

\begin{proof}  Part (1) is from Corollary \ref{nouncountabledecomp} .  Part (3) follows from part (2) by Theorem \ref{unc}.  For part (3), we use part (1), and the kernel of the map $ *_{i\in I}G_i \rightarrow \bigoplus_{i\in I}G_i$ is precisely $\bigcap_{i\in I} \ker(r_i)$ where $r_i$ is the retraction map to $G_i$.  Thus the kernel of $ *_{i\in I}G_i \rightarrow \bigoplus_{i\in I}G_i$ is a countable intersection of analytic subgroups (by part (2)) and therefore analytic.

It remains to prove part (2).  Fix $x\in X$.  It suffices to prove that if $\phi:\pi_1(X ,x) \rightarrow G * H$ is an isomorphism and $r:G*H \rightarrow G$ is the retraction map to $G$ then $\ker(r \circ \phi)$ is analytic.  By Lemma \ref{coverforfreeprod} there exists an open cover $\mathcal{U}_0$  of $X$ by open balls such that $\phi(\pi_1(B_y, y))$ is contained in a conjugate of $G$ or in a conjugate of $H$ for each $B_y\in \mathcal{U}_0$.  Let $\mathcal{U}_1$ be a refinement of $\mathcal{U}_0$ such that any two overlapping elements of $\mathcal{U}_1$ have union contained in an element of $\mathcal{U}_0$.  Let $\mathcal{U}$ be a refinement of $\mathcal{U}_1$ by path connected open sets, with $\mathcal{U}$ countable.  Let $\mathcal{U}_G$ denote those elements of $\mathcal{U}_0$ which map under $\phi$ composed with the inclusion map into a conjugate of $G$, and similarly for $\mathcal{U}_H$.  Thus $\mathcal{U}_G \cup \mathcal{U}_H =\mathcal{U}_0$ and it is possible that $\mathcal{U}_G\cap \mathcal{U}_H \neq \emptyset$.  

Define $\pi_1(\mathcal{U}_0, x)$, $\pi_1(\mathcal{U}_G, x)$ and $\pi_1(\mathcal{U}_H, x)$ as in subsection \ref{Spanier}.  Notice that $\mathcal{U}$ is $2$-set simple rel the quotient map $q:\pi_1(X,x) \rightarrow \pi_1(X,x)/\pi_1(\mathcal{U}_0, x)$.  Then by part (2) of Theorem 7.3 (quoted in our proof of Theorem \ref{freeprodsplitting}) we know that  $\pi_1(X,x)/\pi_1(\mathcal{U}_0, x)$ is countable.  Let $\{w_n\}_{n\in \omega} \subseteq  G * H$ be a countable collection such that $q(\phi^{-1}(\{w_n\}_{n\in \omega})) =  \pi_1(X,x)/\pi_1(\mathcal{U}_0, x)$.  Let $\{h_m\}_{m\in \omega}\subseteq H$ be those elements of $H$ which occur in the words $\{w_n\}_{n\in \omega}$ and let $\{g_n\}_{n\in \omega}$ be defined by letting $g_n = r(w_n)$.

The normal subgroup $\phi^{-1}(\langle\langle \{h_m\}_{m\in \omega}\rangle \rangle)$ is analytic by Theorem \ref{closureprop} part (5) and the normal subgroup $\pi_1(\mathcal{U}_H, x)$ is analytic as well.  Thus the normal subgroup $$\phi^{-1}(\langle\langle \{h_m\}_{m\in \omega}\rangle \rangle)\pi_1(\mathcal{U}_H, x)$$ is analytic by Theorem \ref{closureprop} part (4).  We shall be done if we prove that  $$\langle\langle \{h_m\}_{m\in \omega}\rangle \rangle\phi(\pi_1(\mathcal{U}_H, x)) = \langle \langle H\rangle \rangle$$

The inclusion $\langle\langle \{h_m\}_{m\in \omega}\rangle \rangle\phi(\pi_1(\mathcal{U}_H, x)) \leq \langle \langle H\rangle \rangle$ is self-evident.  For the other inclusion we have 

\begin{center}
$G*H = \{w_n\}_{n\in \omega}\phi(\pi_1(\mathcal{U}_0, x))$

$=  \{w_n\}_{n\in \omega}\phi(\pi_1(\mathcal{U}_H, x))\phi(\pi_1(\mathcal{U}_G, x))$

$=  \{w_n\}_{n\in \omega}\langle\langle \{h_m\}_{m\in \omega}\rangle \rangle\phi(\pi_1(\mathcal{U}_H, x))\phi(\pi_1(\mathcal{U}_G, x))$

$=\{g_n\}_{n\in \omega} \langle\langle \{h_m\}_{m\in \omega}\rangle \rangle\phi(\pi_1(\mathcal{U}_H, x))\phi(\pi_1(\mathcal{U}_G, x))$

$= \langle\langle \{h_m\}_{m\in \omega}\rangle \rangle\phi(\pi_1(\mathcal{U}_H, x)) \{g_n\}_{n\in \omega}\phi(\pi_1(\mathcal{U}_G, x))$

\end{center}

Let $r_H:G*H \rightarrow H$ be the retraction map and $h\in H$.  We know $h = w w'$ where $w\in \langle\langle \{h_m\}_{m\in \omega}\rangle \rangle\phi(\pi_1(\mathcal{U}_H, x))$ and $w'\in \{g_n\}_{n\in \omega}\phi(\pi_1(\mathcal{U}_G, x))$.  Now $h = r_H(h) = r_H(w)r_H(w')= r_H(w)$, so that $H \leq r_H( \langle\langle \{h_m\}_{m\in \omega}\rangle \rangle\phi(\pi_1(\mathcal{U}_H, x)))$.  The group $\phi(\pi_1(\mathcal{U}_H, x))$ is generated by elements of form $h^w$ with $h\in H$, and the same is obviously true of $\langle\langle \{h_m\}_{m\in \omega}\rangle \rangle$.  Thus the subgroup  $\langle\langle \{h_m\}_{m\in \omega}\rangle \rangle\phi(\pi_1(\mathcal{U}_H, x))$ is generated by elements of form $h^w$.  Now $r_H(h^w) = h^{r_H(w)} \in  \langle\langle \{h_m\}_{m\in \omega}\rangle \rangle\phi(\pi_1(\mathcal{U}_H, x))$ for any $h^w\in \langle\langle \{h_m\}_{m\in \omega}\rangle \rangle\phi(\pi_1(\mathcal{U}_H, x))$.  Thus we have shown that $$H \leq  r_H( \langle\langle \{h_m\}_{m\in \omega}\rangle \rangle\phi(\pi_1(\mathcal{U}_H, x))) \leq  \langle\langle \{h_m\}_{m\in \omega}\rangle \rangle\phi(\pi_1(\mathcal{U}_H, x))$$so that $\langle \langle H\rangle\rangle \leq \langle\langle \{h_m\}_{m\in \omega}\rangle \rangle\phi(\pi_1(\mathcal{U}_H, x))$ and we have the other inclusion.
\end{proof}

We use the space $F$ from Example \ref{exampleF} to check aspects of sharpness for Theorem \ref{freeprodPolish}.  Since $\pi_1(F)$ is isomorphic to the free group $F(2^{\aleph_0})$ it is clear that we cannot drop local path connectedness and still assert part (1) of the conclusion.  Local path connectedness is also required for parts (2) and (3) by the example that follows.  It seems unlikely that part (4) holds absent local path connectedness.

\begin{example}\label{notCH}  By 2.12 in \cite{Be} there exists a model $\mathcal{Q}$ of set theory satisfying

\begin{enumerate}  \item \textbf{ZFC}

\item $2^{\aleph_0} = \aleph_3$
\end{enumerate}

We consider the space $F$ in the model $\mathcal{Q}$.  We have $\pi_1(F) \simeq F(2^{\aleph_0})\simeq F(2^{\aleph_0})* F(\aleph_2)$.  The retraction to $F(\aleph_2)$ cannot have analytic kernel by Theorem \ref{CH} part (2).
\end{example}  

It is fascinating to note that the abelian version of Theorem \ref{freeprodPolish} fails.  Recall from Example \ref{productofprojectiveplanes} that we had subgroups $G$ and $H$ of $\pi_1(P^{\omega}, \overline{x})$ such that $\card(H) = \aleph_0$, $G$ was not of nice pointclass with BP and $\pi_1(P^{\omega}, \overline{x}) \simeq H \oplus G$.  Also, $\pi_1(P^{\omega}, \overline{x}) \simeq \bigoplus_{2^{\aleph_0}}\mathbb{Z}/2\mathbb{Z}$ so the index of a direct decomposition into nontrivial groups can fail to be countable even for a Peano continuum.

\end{section}

\begin{section}{Nice Pointclasses}\label{conclusion}
We end with a brief discussion of Polish pointclasses.  We define the projective pointclasses $\Sigma_n^1, \Pi_n^1, \Delta_n^1$ for $n\in \omega \setminus \{0\}$.  We have seen that $\Sigma_1^1$ is the class of analytic sets and $\Pi_1^1$ is the class of coanalytic sets.  Let $\Delta_1^1 = \Sigma_1^1 \cap \Pi_1^1$.  For $n \geq 2$, 

\begin{center}
$\Sigma_n^1$ is the class of continuous images of sets of type $\Pi_{n-1}^1$

$\Pi_{n}^1$ is the class of complements of sets in $\Sigma_n^1$

$\Delta_n^1 = \Pi_{n}^1 \cap \Sigma_n^1$
\end{center}

All Borel sets are easily seen to be of type $\Delta_1^1$.  That $\Delta_1^1$ is precisely the class of Borel sets is a theorem of Suslin (see 14.11 in \cite{Ke}).  The projective pointclasses sit naturally in an array

\begin{center}
$\begin{matrix} & & \Sigma_1^1 & & & & \Sigma_2^1 & &  \cdots\\ \\ \\ \Delta_1^1 & &  & &\Delta_2^1 & &  \cdots\\ \\ \\ & & \Pi_1^1 & & & & \Pi_2^1 & &  \cdots
\end{matrix}$

\end{center}

\noindent with each pointclass containing all pointclasses to the left (e.g. $\Pi_4^1\subseteq  \Delta_5^1$).  In an uncountable Polish space all these inclusions are proper.  Each $\Delta_n^1$ is a $\sigma$-algebra.  Each $\Sigma_n^1$ is closed under taking countable intersections, countable unions and images under continuous maps between Polish spaces.  Each $\Pi_n^1$ is closed under countable unions, countable intersections and continuous preimages.

Each $\Sigma_n^1$ is a nice pointclass and we have already noted that $\Sigma_1^1$ is nice with BP.  Unfortunately, the BP status of the other $\Sigma_n^1$ cannot be decided from \textbf{ZFC} alone.  For example, G\"odel furnished a model of set theory, $L$, in which the following hold:

\begin{enumerate}\item \textbf{ZFC}

\item \textbf{GCH}

\item There exists a $\Delta_2^1$ set which does not have BP.
\end{enumerate}

\noindent (see \cite{G}, \cite{N}, \cite{A}).  Thus it is consistent with \textbf{ZFC} that $\Sigma_1^1$ is the only nice projective pointclass with BP.  Other models of set theory are less stingy with nice pointclasses having BP.  We discuss two situations in which the nice pointclasses with BP are more plentiful: models which satisfy Martin's axiom and \textbf{$\neg$ CH}, and a model of set theory constructed by Shelah \cite{Sh1}.

Martin's axiom (abbreviated \textbf{MA}) is a principle of combinatorial set theory.  We shall not give a formal statement of this principle (the interested reader may consult \cite{Fr}), but state a relevant consequence: if $Z$ is a Polish space then the $\sigma$-algebra of subsets having BP is closed under unions of index less than $2^{\aleph_0}$.  If \textbf{CH} holds then this statement is uninformative (\textbf{MA} is in fact a trivial consequence of \textbf{CH}), but the point is that there exists a model of \textbf{ZFC + MA + $\neg$ CH} (see \cite{ST}).  The principle \textbf{MA} affords more applications of some of the theory developed in this paper.  If $\Po$ is a Polish pointclass we let $S_{\kappa<2^{\aleph_0}}(\Po)$ be the closure of $\Po$ under unions and intersections over indices of cardinality $< 2^{\aleph_0}$.

\begin{proposition}\label{MAprop1}  The following hold:
\begin{enumerate}

\item (\textbf{ZFC + MA})  If $\Po$ is nice with BP then so is $S_{\kappa<2^{\aleph_0}}\Po$.  

\item (\textbf{ZFC + MA + $\neg$ CH}) $S_{\kappa<2^{\aleph_0}}\Sigma_2^1$ is nice with BP.

\end{enumerate}
\end{proposition}

\begin{proof}  For (1) we note that \textbf{MA} makes the $\sigma$-algebra of Baire property sets closed under unions and intersections over cardinals less than $2^{\aleph_0}$.  Thus $S_{\kappa<2^{\aleph_0}}\Po$ has BP, and is obviously nice.

For (2) we recall the theorem of Sierpinski that any $\Sigma_2^1$ set is an $\aleph_1$ union of Borel sets (see 38.8 in \cite{Ke}).  Then by $\neg$ \textbf{CH} we have $\Sigma_2^1 \subseteq S_{\kappa<2^{\aleph_0}}\Sigma_1^1$ and we apply (1), using the obvious fact that the operation $S_{\kappa<2^{\aleph_0}}$ is idempotent.
\end{proof}

The following proposition illustrates some consequences that can be derived from \textbf{ZFC + MA + $\neg$ CH}:

\begin{proposition}\label{consequencesofMA}   (\textbf{ZFC + MA + $\neg$ CH}) Let $X$ be a Peano continuum.  The quotient $\pi_1(X)/G$ is either finitely generated or of cardinality $2^{\aleph_0}$ for the following $G$:

\begin{enumerate}  \item The center $G = Z\pi_1(X)$.

\item $G = \langle\langle [\{l_{\alpha}\}_{\alpha<\aleph_1}]\rangle\rangle$ where $\{l_{\alpha}\}_{\alpha<\aleph_1}$ is any collection of loops of cardinality $\aleph_1$.

\item $G = (\pi_1(X))^{(\alpha)}$ for any $\alpha<2^{\aleph_0}$.

\end{enumerate}
\end{proposition}

\begin{proof}  For (1), if $G$ is open then we are done by Theorem \ref{Shelahgen}.  If $G$ is not open then by Lemma \ref{openorsequence} there is a point $y\in X$ and sequence of loops $\{l_n\}_{n\in \omega}$ with $\diam(l_n)\searrow 0$ and $[l_n]\notin G$.  Notice that

\begin{center}  $l\in \bigcup Z\pi_1(X) \Longleftrightarrow (\forall l'\in L_y)(\exists H\in H_y)[H$ homotopes $l*l'$ to $l'*l]$

\end{center}

\noindent which illustrates that $ \bigcup Z\pi_1(X)$ is $\Pi_2^1$.  By Proposition \ref{MAprop1} we know that $\Sigma_2^1$ has BP and so $\Pi_2^1$ has BP.  Since  $\Pi_2^1$ is also closed under continuous preimages we have part (1) by Lemma \ref{uncountable}.

For part (2) we notice that $\{l_{\alpha}\}_{\alpha<\aleph_1}$ is $S_{\kappa<2^{\aleph_0}}\Sigma_1^1$ and so we can directly apply Theorem \ref{Shelahgen} since $S_{\kappa<2^{\aleph_0}}\Sigma_1^1$ is nice with BP.  For part (3) one can prove by induction over all ordinals less than $2^{\aleph_0}$ that $ (\pi_1(X))^{(\alpha)}$ is of type $S_{\kappa<2^{\aleph_0}}\Sigma_1^1$.
\end{proof}

Very generous applications of our theory can be derived from 7.17 in \cite{Sh1}:

\begin{corollary*}  There exists a model $\mathcal{R}$ of set theory in which the following hold:

\begin{enumerate}\item \textbf{ZFC}

\item All projective subsets of $\{0,1\}^{\omega}$ have BP.
\end{enumerate}

\end{corollary*}

In $\mathcal{R}$ the conclusions of Theorems \ref{unc}, \ref{Shelahgen}, and \ref{ascendingnormal} apply to subgroups definable from first-order formulas in conjunction with iterations of countable set operations applied to subgroups already known to be of a projective pointclass.  We illustrate with an example.  In $\mathcal{R}$ if $X$ is a Peano continuum and $x\in X$ then letting $G$ be the normal subgroup generated by the set of elements that are not a cube of a central element:

\begin{center}
$G = \langle\langle \{g\in \pi_1(X, x):\neg(\exists h \in Z\pi_1(X, x))[h^3 = g] \}\rangle\rangle$
\end{center}

\noindent we have that $G$ is a $\Sigma_4^1$ subgroup and so the quotient $\pi_1(X, x)/G$ is either finitely generated or of cardinality $2^{\aleph_0}$.

\end{section}

\bibliographystyle{amsplain}

\end{document}